\documentclass[12pt]{amsart}
\usepackage{amsfonts,amssymb,amsmath,amsthm}

\usepackage[colorlinks, linkcolor=blue, citecolor=blue, urlcolor=blue,
pagebackref,hypertexnames=false]{hyperref}
\usepackage{enumerate}
\usepackage{geometry}
\usepackage{txfonts}

\allowdisplaybreaks
\numberwithin{equation}{section}
\theoremstyle{plain}
\newtheorem{theorem}{Theorem}[section]
\newtheorem{proposition}[theorem]{Proposition}
\newtheorem{lemma}[theorem]{Lemma}
\newtheorem{corollary}[theorem]{Corollary}

\newcommand{\calC}{\mathcal C}
\newcommand{\calO}{\mathcal O}
\newcommand{\one}{\mathbf 1}
\newcommand{\eps}{\varepsilon}
\newcommand{\Real}{\operatorname{Re}}

\begin{document}

\title[Fefferman--Stein inequality in the Dunkl setting]
{A Fefferman--Stein inequality for the Dunkl Poisson semigroup\\ and its chamber-lifted formulation}

\author[Y. Chen]{Yuying Chen}
\address{School of Mathematical Sciences, South China Normal University, Guangzhou, 510631, P.R. China}
\email{2022021927@m.scnu.edu.cn}

\author[Ya. Han]{Yanchang Han}
\address{School of Mathematical Sciences, South China Normal University, Guangzhou, 510631, P.R. China}
\email{hanych@scnu.edu.cn}

\author[Yo. Han]{Yongsheng Han}
\address{Department of Mathematics, Auburn University, AL 36849-5310, USA}
\email{hanyong@auburn.edu}

\author[J. Li]{Ji Li}
\address{School of Mathematical and Physical Sciences, Macquarie University, NSW 2109, Australia}
\email{ji.li@mq.edu.au}

\author[L. Wu]{Liangchuan Wu}
\address{School of Mathematical Sciences, Anhui University, Hefei, 230601, P.R. China}
\email{wuliangchuan@ahu.edu.cn}

\begin{abstract}
We prove a Fefferman--Stein good-$\lambda$ inequality for the Dunkl Poisson semigroup associated with a finite reflection
group and a non-negative multiplicity function.  For arbitrary complex-valued $f\in C_c^\infty(\mathbb R^N)$, with no
$G$-invariance assumption, it compares the orbit-conical non-tangential maximal function $\mathcal N_P^\beta f$ with
the area function $\mathcal S_Pf$ formed from the full space-time Dunkl carr\'e du champ, including its
reflection-difference energy.  The main obstruction is that a general cut-off creates wall differences not controlled
by the local Euclidean-gradient product identity.  The good set
$E=\{x:\mathcal N_P^\beta f(x)\le\lambda\}$ is $G$-invariant; by the equivariance of the Poisson semigroup, so is
$a=\varphi(P_t\one_E)$, and hence all reflection differences of the cut-off vanish.  Poisson maximal and tail estimates,
together with the $L^2$ Littlewood--Paley estimate for $P_t\one_{E^c}$, then yield the desired distribution inequality.
Its integrated form gives maximal-to-area estimates for every $0<p<2$ and endpoint $H^1$-to-$L^1$ bounds for the
orbit-conical and Euclidean-conical intrinsic area functions.  For chamber lifts of globally smooth data, the inequality
has an equivalent formulation on a fundamental chamber, where orbit cones become Euclidean cones and the reflection
energy becomes a finite wall coupling.  Combined with the known semigroup square-function characterization, these
bounds characterize the Dunkl Poisson maximal Hardy space among $L^1(d\omega)$ data.
\end{abstract}

\subjclass[2020]{42B25, 42B30, 42B35, 33C52}
\keywords{Dunkl operators, Fefferman--Stein inequality, square functions, non-tangential maximal functions, chamber lifting,
reflection groups}

\maketitle

\section{Introduction}

In the classical upper half-space, the non-tangential maximal function $u^*$ and the Lusin area integral $S(u)$
satisfy a distribution estimate of the form
$$
    \left|\{x:S(u)(x)>\lambda\}\right|  \le C\left|\{x:u^*(x)>\lambda\}\right|
    +\frac{C}{\lambda^2}\int_0^\lambda s\left|\{x:u^*(x)>s\}\right|\,ds.
$$
This theorem of Fefferman and Stein \cite{FS} underlies maximal-to-area estimates and Hardy-space characterizations;
see also Stein \cite{S}.  Merryfield \cite{M} obtained the same distribution estimate without the surface
approximation used in \cite{FS}.  His argument constructs the Littlewood--Paley test function from the
maximal-function test function through a Cauchy--Riemann system.  J. Li \cite{Li} later established the corresponding
estimate on general Shilov boundaries through a Poisson-extension and bootstrapping argument that avoids this
auxiliary-function construction.

Rational Dunkl analysis couples Euclidean differentiation with the action of a finite reflection group.
The associated measure is doubling, but the Poisson kernel is localized by the orbit distance, whereas the
Dirichlet energy contains differences across reflecting walls.  This mixed local--nonlocal structure prevents a
direct transfer of the classical maximal-to-area argument to the Dunkl setting.  The underlying operators,
semigroups, transform theory, singular integrals, multiplier estimates, and
Hardy spaces have been studied extensively; see
\cite{D1,D2,D3,RM,ADH,Bui,DH1,DH2,DH3,DHFSBMO,THHLL1,THHLL2}.

For the notation in the main theorem, let $R$ be a reduced root system, $G$ its reflection group, $\kappa\ge0$ a
$G$-invariant multiplicity function, $d\omega$ the associated Dunkl measure, and $\Delta_\kappa$ the Dunkl Laplacian.
Set $P_t=e^{-t\sqrt{-\Delta_\kappa}}$, fix a positive subsystem $R_+$, and denote by $\sigma_\alpha$ the reflection
across $\alpha^\perp$.  Write $\langle\cdot,\cdot\rangle$ and $\|\cdot\|$ for the Euclidean inner product and norm,
respectively.  For $x,y\in\mathbb R^N$, set
$$
        d(x,y):=\min_{\sigma\in G}\|x-\sigma y\|.
$$
For $t>0$, write
$$
        V(x,t):=\omega(B(x,t)),
$$
where $B(x,t)$ is the Euclidean ball.

Several neighboring square-function theories are available.  The conical square function in \cite{ADH} is formed
from the semigroup derivative $Q_t=t\sqrt{-\Delta_\kappa}e^{-t\sqrt{-\Delta_\kappa}}$, while the Lusin-area
characterizations in \cite{CHHT} arise from a semi-discrete reproducing formula.  Vertical square functions
containing the spatial Dunkl carr\'e du champ were treated in \cite{LiZhao,DHLP}, and related heat-semigroup
operators were studied in \cite{ABFR,LiWeak}.  Deleaval's Fefferman--Stein inequality for the
$\mathbb Z_2^d$ Dunkl maximal operator \cite{Deleaval} concerns a different maximal-operator problem.

Area integrals defined through Dunkl generalized translations were studied locally by J. Jiu and Z. Li
\cite{JiuLiLocal}.  Jiu \cite{Jiu} proved the global maximal-to-area estimate for arbitrary $\kappa$-harmonic
functions.  Under a vanishing-at-infinity assumption, the reverse estimate holds for $G$-invariant functions
associated with an arbitrary finite reflection group, and for arbitrary functions when $G=\mathbb Z_2^d$.
These functionals differ from the continuous orbit-conical integral of the pointwise space-time carr\'e du champ
considered here and therefore do not yield the distribution estimate below for the full intrinsic energy and
arbitrary complex, non-$G$-invariant data.

The obstruction appears already in the product rule.  Let $u=P_tf$, and let $a$ be a cut-off.  For the reflection
corresponding to $\alpha$,
$$
    a(x)u(x)-a(\sigma_\alpha x)u(\sigma_\alpha x)  =a(x)\bigl(u(x)-u(\sigma_\alpha x)\bigr)
    +u(\sigma_\alpha x)\bigl(a(x)-a(\sigma_\alpha x)\bigr).
$$
The last term is a wall difference created by the cut-off.  It is not controlled by the local Euclidean-gradient term
in the classical product identity.  Hence this direct localization does not close unless the cut-off has vanishing
reflection differences.

For $u(y,t)=P_tf(y)$, the intrinsic energy density is
$$
    \Gamma_{\kappa,t}(u)(y,t)  =|\partial_tu(y,t)|^2+|\nabla_yu(y,t)|^2
    +\sum_{\alpha\in R_+}\kappa(\alpha)   \frac{|u(y,t)-u(\sigma_\alpha y,t)|^2}{\langle\alpha,y\rangle^2}.
$$
The quotient is interpreted by its continuous extension on the reflecting hyperplanes.  Define
$$
    \mathcal N_P^\beta f(x)=\sup_{d(x,y)<\beta t}|u(y,t)|,
$$
and
$$
    \mathcal S_Pf(x)   =\left(   \int_0^\infty\int_{d(x,y)<t}   \Gamma_{\kappa,t}(u)(y,t)\,
    \frac{t\,dt\,d\omega(y)}{V(x,t)}    \right)^{1/2}.
$$

\begin{theorem}\label{thm:main}
There exist a structural aperture $\beta>1$ and a constant $C>0$ such that, for every complex-valued
$f\in C_c^\infty(\mathbb R^N)$ and every $\lambda>0$,
\begin{align}\label{eq:main}
    \omega\{x\in\mathbb R^N:\mathcal S_Pf(x)>\lambda\}
     &\le C\,\omega\{x\in\mathbb R^N:\mathcal N_P^\beta f(x)>\lambda\}  +\frac{C}{\lambda^2}
    \int_{\{\mathcal N_P^\beta f\le\lambda\}}   \bigl(\mathcal N_P^\beta f(x)\bigr)^2\,d\omega(x).
\end{align}
The constants depend only on $N$, the root system, the multiplicity function, and the doubling constants of
$(\mathbb R^N,d\omega)$; no invariance of $f$ under the reflection group is assumed.
\end{theorem}

The elementary truncation identity for distribution functions shows that \eqref{eq:main} is equivalent, after a
change in $C$, to
\begin{align}\label{eq:layer}
    \omega\{x\in\mathbb R^N:\mathcal S_Pf(x)>\lambda\}
    &\le C\,\omega\{x\in\mathbb R^N:\mathcal N_P^\beta f(x)>\lambda\}  +\frac{C}{\lambda^2}\int_0^\lambda
    s\,\omega\{x\in\mathbb R^N:\mathcal N_P^\beta f(x)>s\}\,ds.
\end{align}
Integrating \eqref{eq:layer}, we obtain
$$
    \|\mathcal S_Pf\|_{L^p(d\omega)}
    \le C_p\|\mathcal N_P^\beta f\|_{L^p(d\omega)},
    \qquad 0<p<2.
$$
The smoothness assumption justifies the localized energy identity and its boundary traces.  It is removed at the
Hardy-space endpoint by density and lower semicontinuity; the distribution estimate itself is stated only for smooth
compactly supported boundary data.

The good set itself supplies the invariant cut-off.  Although $f$ need not be $G$-invariant,
$\mathcal N_P^\beta f$ is.  Hence
\begin{equation*}
        E=\{x:\mathcal N_P^\beta f(x)\le\lambda\},   \qquad    v=P_t\one_E,
        \qquad    a=\varphi(v)
\end{equation*}
are $G$-invariant in the spatial variable.  In particular,
\begin{equation*}
        a(y,t)-a(\sigma_\alpha y,t)=0,    \qquad \alpha\in R,
\end{equation*}
and the product rule for $au$, $u=P_tf$, contains no additional reflection term.  The Poisson maximal and tail
estimates make $a$ equal to one over the thick good-set tent and confine its derivatives to the enlarged tent over
$E$.  The remaining error is bounded by the $L^2$ Littlewood--Paley estimate for $P_t\one_{E^c}$.

The chamberwise decomposition used to prove the sharp $L^p$ bounds for Dunkl area integrals in \cite{DunklArea}
is the starting point of the formulation below.  In \cite{HanLeeLiSawyerWuCalderon}, it was developed into the
finite-coordinate chamber lifting: all reflected values are retained as separate coordinates on one fundamental
chamber, so the orbit diagonal becomes the ordinary diagonal and a full-space operator becomes a finite matrix.
The same lift was later used for the $L^2$ two-weight testing problem for Dunkl--Poisson integrals in \cite{GLWW}.
For non-radial Dunkl multipliers, the same lifting yields a finite-matrix chamber criterion; for $A_1^N$,
a Walsh--Bessel argument verifies this criterion for wall-separated symbol pieces in \cite{CLWW}.
The common step is geometric, whereas the analytic argument remains problem-dependent.  Here the lifting enters only
after the full-space good-$\lambda$ inequality has been proved; it identifies orbit cones with Euclidean cones on
chamber representatives and rewrites the reflection energy as a finite wall coupling.

For the present formulation, fix a closed fundamental chamber $\calC$ and enumerate
$G=\{\sigma_1,\ldots,\sigma_m\}$.  The lift
$$
    Uf(x)=\left(f(\sigma_1x),\ldots,f(\sigma_mx)\right),   \qquad x\in\calC,
$$
retains the orbit values as distinct fiber coordinates and preserves $L^p$ norms.  For $x,y\in\calC$,
Lemma~\ref{lem:close} states that $d(\sigma_\rho x,\sigma_\tau y)=\|x-y\|$, while the reflection energy becomes a finite
wall coupling between the coordinates.  Thus Theorem~\ref{thm:ch} is equivalent to Theorem~\ref{thm:main}, up to
structural constants, on the lifted range $F=Uf$ with $f\in C_c^\infty(\mathbb R^N)$.  No theorem is asserted for
chamberwise smooth vector data that lack the cross-wall compatibility of a globally smooth function.

Our second main result is the maximal-to-area estimate at the Hardy-space endpoint for the intrinsic area
functionals themselves.  Let $S_{P,\mathrm{euc}}$ be the same intrinsic energy integrated over the Euclidean cone
$\|x-y\|<t$.  As proved in Theorem~\ref{thm:h1},
$$
    \|\mathcal S_Pf\|_{L^1(d\omega)}   +\|S_{P,\mathrm{euc}}f\|_{L^1(d\omega)}
    \le C\|f\|_{H^1_{\mathrm{max},P}},   \qquad f\in H^1_{\mathrm{max},P}.
$$
Here $H^1_{\mathrm{max},P}$ denotes the Poisson maximal Hardy space.
Let $S_Q$ denote the Euclidean-conical square function associated with
$Q_t=t\sqrt{-\Delta_\kappa}\,P_t$.  For $f\in L^1(d\omega)$, the pointwise chain
$S_Qf\le S_{P,\mathrm{euc}}f\le\mathcal S_Pf$ and the semigroup square-function theorem in \cite{ADH} imply
$$
   f\in H^1_{\mathrm{max},P}  \quad\Longleftrightarrow\quad
    S_{P,\mathrm{euc}}f\in L^1(d\omega)   \quad\Longleftrightarrow\quad
     \mathcal S_Pf\in L^1(d\omega),
$$
with equivalence of the three norms.  The present argument supplies the maximal-to-area implication; the reverse
norm bound uses the known $S_Q$ characterization, and no reverse distribution-function estimate is asserted.

This paper is organised as follows.  Sections~\ref{sec:pre} and~\ref{sec:pois} develop the Dunkl energy and
the orbit-geometric semigroup estimates.  Sections~\ref{sec:loc} and~\ref{sec:eng} construct the invariant
cut-off and prove the good-$\lambda$ inequality.
Section~\ref{sec:ch} presents the chamber formulation, and Section~\ref{sec:euc} passes from orbit cones to
Euclidean cones.  Section~\ref{sec:h1} proves the Hardy-space endpoint theorem and records the $L^p$ consequence.

\section{The Dunkl differential structure}\label{sec:pre}

We first fix the reflection data and then establish the carr\'e du champ identity and the invariant product and chain
rules used below.

\subsection{Root systems, reflections, and the Dunkl measure}

Let $R \subset \mathbb{R}^N \setminus \{0\}$ be a reduced root system.  We normalize
$$
        \|\alpha\|^2 = 2, \qquad \alpha \in R,
$$
and all structural constants below correspond to this normalization.
For each $\alpha \in R$, let $\sigma_\alpha$ be the reflection across the hyperplane $\alpha^\perp$, defined by
\begin{equation}\label{eq:refl}
        \sigma_\alpha x = x - 2\frac{\langle x, \alpha \rangle}{\|\alpha\|^2}\alpha.
\end{equation}
Let $G$ be the finite reflection group generated by $\{\sigma_\alpha : \alpha \in R\}$.
We fix a positive subsystem $R_+$, and let $\kappa : R \to [0, \infty)$ be a $G$-invariant multiplicity function.
The associated Dunkl weight and measure are
$$
        w(x) = \prod_{\alpha \in R_+} |\langle x, \alpha \rangle|^{2\kappa(\alpha)},
        \qquad d\omega(x) = w(x) \, dx.
$$
Denote the union of the reflecting hyperplanes by
$$
        \mathcal W:=\bigcup_{\alpha\in R}\alpha^\perp.
$$
Since $R$ is finite and $d\omega$ is absolutely continuous with respect to Lebesgue measure,
$\omega(\mathcal W)=0$.  Moreover, $w>0$ on $\mathbb R^N\setminus\mathcal W$, so $d\omega$ has full support.
Equivalently,
\begin{equation*}
        w(x) = \prod_{\alpha \in R} |\langle x, \alpha \rangle|^{\kappa(\alpha)},
\end{equation*}
since roots occur in pairs $\{\alpha, -\alpha\}$ with $\kappa(-\alpha) = \kappa(\alpha)$.
Geometrically, $|\langle x,\alpha\rangle|$ is the unnormalized normal coordinate of $x$
relative to the wall $\alpha^\perp$; more precisely,
$$
        \mathrm{dist}(x, \alpha^\perp) = \frac{|\langle x, \alpha \rangle|}{\|\alpha\|}.
$$
This convention fixes the normalization used in the wall terms below.
Throughout the paper, $B(x,r)$ denotes the Euclidean ball centered at $x$ with radius $r$.

For $\sigma\in G$, the $G$-invariance of $\kappa$ implies
\begin{align*}
    w(\sigma x)   &=\prod_{\alpha\in R}|\langle x,\sigma^{-1}\alpha\rangle|^{\kappa(\alpha)}
    =\prod_{\beta\in R}|\langle x,\beta\rangle|^{\kappa(\sigma\beta)}   =w(x).
\end{align*}
Since $\sigma$ is orthogonal, a change of variables then shows, for every non-negative measurable $\Phi$, that
$$
        \int_{\mathbb{R}^N} \Phi(\sigma x) \, d\omega(x) = \int_{\mathbb{R}^N} \Phi(x) \, d\omega(x).
$$

Let
$$
        \gamma=\sum_{\alpha\in R_+}\kappa(\alpha),   \qquad    \mathbf N=N+2\gamma.
$$
Since $w(sx)=s^{2\gamma}w(x)$, a change of variables leads to
\begin{align*}
    \omega(B(sx,sr))   &=\int_{B(sx,sr)}w(z)\,dz   =s^N\int_{B(x,r)}w(sy)\,dy
    =s^{\mathbf N}\omega(B(x,r)),   \qquad s,r>0.
\end{align*}
It follows from \cite[(3.1)--(3.2), p.~7]{ADH}  that 
$$
    \omega(B(x,r))   \simeq r^N\prod_{\alpha\in R_+}   \bigl(|\langle x,\alpha\rangle|+r\bigr)^{2\kappa(\alpha)}
$$
and, for $0<r\le R$,
$$
        C^{-1}\left(\frac Rr\right)^N    \le\frac{\omega(B(x,R))}{\omega(B(x,r))}    \le C\left(\frac Rr\right)^{\mathbf N}.
$$
In particular, $d\omega$ is doubling.

\subsection{Dunkl operators and the Dunkl Laplacian}

For $\xi \in \mathbb{R}^N$, the Dunkl operator is defined by
$$
        T_\xi f(x)   = \partial_\xi f(x)   + \sum_{\alpha\in R_+} \kappa(\alpha) \langle \alpha, \xi \rangle \,
        \frac{f(x) - f(\sigma_\alpha x)}{\langle \alpha, x \rangle}.
$$
For convenience, we write $T_j = T_{e_j}$ for $j=1, \ldots, N$,
where $\{e_j\}_{j=1}^N$ is the standard orthonormal basis of $\mathbb{R}^N$.
The corresponding Dunkl Laplacian is
$$
        \Delta_\kappa = \sum_{j=1}^N T_j^2.
$$
When acting on smooth functions, it has the following explicit form:
\begin{equation}\label{eq:lap}
        \Delta_\kappa f(x)  = \Delta f(x)  + 2\sum_{\alpha\in R_+} \kappa(\alpha)
        \left(  \frac{\partial_\alpha f(x)}{\langle \alpha, x \rangle}  
        - \frac{f(x) - f(\sigma_\alpha x)}{\langle \alpha, x \rangle^2}   \right),
\end{equation}
where $\Delta$ denotes the standard Euclidean Laplacian,
and $\partial_\alpha f = \langle \nabla f, \alpha \rangle$ is the directional derivative along $\alpha$.

We also use the upper half-space
$\mathbb R^{N+1}_+:=\mathbb R^N\times(0,\infty)$.  Set
$$
        \Delta_{\kappa,t} = \partial_t^2 + \Delta_\kappa.
$$

\subsection{The carr\'e du champ}

All functions under consideration may be complex-valued.
Accordingly, we introduce the sesquilinear energy form
$$
    \Gamma_\kappa(F,H)(x)   = \nabla F(x) \cdot \overline{\nabla H(x)}   + \sum_{\alpha\in R_+} \kappa(\alpha) \,
    \frac{\bigl(F(x) - F(\sigma_\alpha x)\bigr) \overline{\bigl(H(x) - H(\sigma_\alpha x)\bigr)}}{\langle \alpha, x \rangle^2}.
$$
We write $\Gamma_\kappa(F) = \Gamma_\kappa(F,F)$, which satisfies $\Gamma_\kappa(F) \ge 0$ pointwise.
The extended upper-half-space energy is
$$
        \Gamma_{\kappa,t}(U,V)  = \partial_t U \, \overline{\partial_t V} + \Gamma_\kappa(U,V),
        \qquad    \Gamma_{\kappa,t}(U) = \Gamma_{\kappa,t}(U,U).
$$
In particular,
$$
    \Gamma_{\kappa,t}(U)(x,t)  = |\partial_t U(x,t)|^2 + |\nabla_x U(x,t)|^2
    + \sum_{\alpha\in R_+} \kappa(\alpha) \,
    \frac{|U(x,t) - U(\sigma_\alpha x,t)|^2}{\langle \alpha, x \rangle^2}.
$$
For smooth $U$, the apparent wall singularity is removable.  Indeed, if $y\in\alpha^\perp$ and $x\to y$, then
$$
        \frac{U(x,t)-U(\sigma_\alpha x,t)}{\langle\alpha,x\rangle}
        \longrightarrow \partial_\alpha U(y,t).
$$
Thus the quotient extends continuously across $\alpha^\perp$ and is locally bounded on compact subsets;
its values on the $\omega$-null reflecting hyperplanes do not affect the integrals below.

The following product identity applies to complex-valued functions.

\begin{lemma}
For all smooth complex-valued functions $F$ and $H$ on $\mathbb{R}^N$,
\begin{equation}\label{eq:gprod}
        \Delta_\kappa(F\overline{H})-F\Delta_\kappa\overline{H}  -\overline{H}\Delta_\kappa F    =2\Gamma_\kappa(F,H).
\end{equation}
Consequently, every smooth complex-valued function $U=U(x,t)$ satisfies
\begin{equation}\label{eq:gsq}
        \Delta_{\kappa,t}|U|^2    -2\Real\bigl(\overline{U}\,\Delta_{\kappa,t}U\bigr)    =2\Gamma_{\kappa,t}(U).
\end{equation}
\end{lemma}

\begin{proof}
Away from the reflecting hyperplanes, write $F_\alpha=F\circ\sigma_\alpha$ and
$H_\alpha=H\circ\sigma_\alpha$.  Substitution into \eqref{eq:lap}, followed by the Euclidean product rule, leads to
\begin{align*}
    &  \Delta_\kappa(F\overline{H})-F\Delta_\kappa\overline{H}    -\overline{H}\Delta_\kappa F\\
    =\,&2\nabla F\cdot\overline{\nabla H}   +2\sum_{\alpha\in R_+}\kappa(\alpha)   \frac{\partial_\alpha(F\overline{H})-F\partial_\alpha\overline{H}
       -\overline{H}\partial_\alpha F}{\langle\alpha,x\rangle}\\
    &\quad    +2\sum_{\alpha\in R_+}\kappa(\alpha)    \frac{-(F\overline{H}-F_\alpha\overline{H_\alpha})
       +F(\overline{H}-\overline{H_\alpha})
       +\overline{H}(F-F_\alpha)}{\langle\alpha,x\rangle^2}\\
    =\,&2\nabla F\cdot\overline{\nabla H}   +2\sum_{\alpha\in R_+}\kappa(\alpha)   \frac{(F-F_\alpha)\overline{(H-H_\alpha)}}{\langle\alpha,x\rangle^2}
    =2\Gamma_\kappa(F,H).
\end{align*}
Continuity across the reflecting hyperplanes extends \eqref{eq:gprod} to all of $\mathbb R^N$.

For $F=H=U$, append the ordinary $t$-product identity to \eqref{eq:gprod}:
\begin{align*}
    &\Delta_{\kappa,t}|U|^2  -2\Real\bigl(\overline{U}\,\Delta_{\kappa,t}U\bigr)\\
    =\,&\Delta_\kappa|U|^2-U\Delta_\kappa\overline{U}  -\overline{U}\Delta_\kappa U
      +\partial_t^2|U|^2-U\partial_t^2\overline{U}   -\overline{U}\partial_t^2U\\
    =\,&2\Gamma_\kappa(U)+2|\partial_tU|^2   =2\Gamma_{\kappa,t}(U).
\end{align*}
\end{proof}

For $G$-invariant factors, the product and chain rules retain their Euclidean form.

\begin{lemma}
Let $U=U(x,t)$ be smooth and complex-valued.  If $a=a(x)$ is smooth,
real-valued, and $G$-invariant, then
\begin{equation}\label{eq:prod}
        \Delta_\kappa(aU)   =a\Delta_\kappa U+U\Delta_\kappa a+2\nabla a\cdot\nabla U.
\end{equation}
If $A=A(x,t)$ is smooth, real-valued, and $G$-invariant in $x$, then
\begin{equation}\label{eq:prodt}
        \Delta_{\kappa,t}(AU) =A\Delta_{\kappa,t}U+U\Delta_{\kappa,t}A    +2\nabla_{x,t}A\cdot\nabla_{x,t}U.
\end{equation}
If $V=V(x,t)$ is smooth, real-valued, and $G$-invariant in $x$, then every
$\Phi\in C^2(\mathbb{R})$ satisfies
\begin{equation}\label{eq:chain}
        \Delta_{\kappa,t}\Phi(V)    =\Phi'(V)\Delta_{\kappa,t}V+\Phi''(V)|\nabla_{x,t}V|^2,
\end{equation}
where $\nabla_{x,t}=(\nabla_x,\partial_t)$.
\end{lemma}

\begin{proof}
Since $a\circ\sigma_\alpha=a$, the reflection part of $\Gamma_\kappa(U,a)$ vanishes.  Equation~\eqref{eq:gprod}
with $F=U$ and $H=a$ therefore shows
\begin{align*}
     \Delta_\kappa(aU)-a\Delta_\kappa U-U\Delta_\kappa a   &=2\Gamma_\kappa(U,a)
    =2\nabla U\cdot\nabla a   =2\nabla a\cdot\nabla U,
\end{align*}
which is \eqref{eq:prod}.  At each fixed $t$, applying the spatial identity to $A(\cdot,t)$ and expanding the ordinary
$t$-derivative, we obtain
\begin{align*}
    \Delta_{\kappa,t}(AU)   &=\Delta_\kappa(AU)+\partial_t^2(AU)\\
    &=A\Delta_\kappa U+U\Delta_\kappa A   +2\nabla_xA\cdot\nabla_xU
      +A\partial_t^2U+U\partial_t^2A   +2\partial_tA\,\partial_tU\\
    &=A\Delta_{\kappa,t}U+U\Delta_{\kappa,t}A   +2\nabla_{x,t}A\cdot\nabla_{x,t}U.
\end{align*}

For \eqref{eq:chain}, the reflection differences of both $V$ and $\Phi(V)$ vanish.  Away from the reflecting
hyperplanes, we have
\begin{align*}
    \Delta_{\kappa,t}\Phi(V)   &=\partial_t^2\Phi(V)+\Delta\Phi(V)   +2\sum_{\alpha\in R_+}\kappa(\alpha)
       \frac{\partial_\alpha\Phi(V)}{\langle\alpha,x\rangle}\\
    &=\Phi'(V) \,\bigg(   \partial_t^2V+\Delta V   +2\sum_{\alpha\in R_+}\kappa(\alpha)
        \frac{\partial_\alpha V}{\langle\alpha,x\rangle}  \bigg)  +\Phi''(V)|\nabla_{x,t}V|^2\\
    &=\Phi'(V)\Delta_{\kappa,t}V     +\Phi''(V)|\nabla_{x,t}V|^2.
\end{align*}
Continuity covers the reflecting hyperplanes.
\end{proof}

\section{Orbit geometry and Poisson estimates}\label{sec:pois}

The localization uses two orbit-geometric estimates: maximal control near the thick good set and Poisson decay
outside its enlarged tent.

\subsection{Orbit distance and orbit balls}

Define the orbit distance by
$$
        d(x,y)=\min_{\sigma\in G}\|x-\sigma y\|     =\min_{\sigma\in G}\|\sigma x-y\|.
$$
For $\rho,\tau\in G$, we have $ d(\rho x,\tau y)=d(x,y)$, and
$$
        d(x,y)=0\quad\Longleftrightarrow\quad y=\sigma x   \quad\text{for some }\sigma\in G.
$$
Choose $\sigma,\tau\in G$ such that
$d(x,y)=\|x-\sigma y\|$ and $d(y,z)=\|y-\tau z\|$.  Then
\begin{align*}
        d(x,z)   &\le \|x-\sigma\tau z\|    \le \|x-\sigma y\|+\|\sigma y-\sigma\tau z\|
         =d(x,y)+d(y,z).
\end{align*}
Thus $d$ is a $G$-invariant pseudometric on $\mathbb R^N$ and induces a metric on $\mathbb R^N/G$.

For $r>0$, let
$$
        \calO(x,r):=\{y\in\mathbb R^N:d(x,y)<r\}    =\bigcup_{\sigma\in G}B(\sigma x,r),
$$
and
$$
        V(x,r):=\omega(B(x,r)),    \qquad    V(x,y,r):=\max\{V(x,r),V(y,r)\}.
$$
The $G$-invariance of $d\omega$ implies
\begin{equation}\label{eq:ovol}
        V(x,r)   \le\omega(\calO(x,r))     \le\sum_{\sigma\in G}V(\sigma x,r)   =|G|V(x,r).
\end{equation}
Consequently,
$$
        \omega(\calO(x,2r))    \le |G|V(x,2r)  \lesssim |G|V(x,r)    \le |G|\omega(\calO(x,r)),
$$
so the orbit balls are doubling.

Finally, suppose that $d(x,y)<r$, and choose $\sigma\in G$ such that
$\|\sigma x-y\|<r$.  Since
$$
        B(y,r)\subset B(\sigma x,2r),   \qquad    B(\sigma x,r)\subset B(y,2r),
$$
the $G$-invariance of $d\omega$ and doubling show that
\begin{align*}
        V(y,r)    &\le V(\sigma x,2r)    \lesssim V(\sigma x,r)
        =V(x,r)   \le V(y,2r)   \lesssim V(y,r).
\end{align*}
Hence $V(x,r)\simeq V(y,r)$, with constants depending only on the doubling constant.

\subsection{The Dunkl Poisson semigroup and kernel estimates}

Let $\mathcal F_\kappa$ denote the Dunkl transform.  By
\cite[Lemma~2.6 (2) and Theorem~2.6 (1)--(2), pp.~109--110]{RM},
$\mathcal F_\kappa$ is unitary on $L^2(d\omega)$, preserves
$\mathcal S(\mathbb R^N)$, and satisfies
$$
        \mathcal F_\kappa(-\Delta_\kappa f)(\xi)    =\|\xi\|^2\mathcal F_\kappa f(\xi),   \qquad f\in\mathcal S(\mathbb R^N).
$$
We use $-\Delta_\kappa$ for the non-negative self-adjoint realization characterized by
$$
        D(-\Delta_\kappa)     :=\bigl\{f\in L^2(d\omega):     \|\xi\|^2\mathcal F_\kappa f\in L^2(d\omega)\bigr\},
        \qquad   \mathcal F_\kappa(-\Delta_\kappa f)(\xi)      =\|\xi\|^2\mathcal F_\kappa f(\xi).
$$
Since $\omega(\{0\})=0$, every $f\in D(-\Delta_\kappa)$ satisfies
$$
    -\Delta_\kappa f=0  \ \Longrightarrow\    \|\xi\|^2\mathcal F_\kappa f(\xi)=0\quad d\omega\text{-a.e.}
    \ \Longrightarrow\     \mathcal F_\kappa f=0    \ \Longrightarrow\   f=0.
$$
Hence
$$
        \ker_{L^2(d\omega)}(-\Delta_\kappa)=\{0\}.
$$

For $t>0$, define
$  P_t=e^{-t\sqrt{-\Delta_\kappa}}.$
Equivalently,
$$
        \mathcal F_\kappa(P_tf)(\xi)    =e^{-t\|\xi\|}\mathcal F_\kappa f(\xi),    \qquad f\in L^2(d\omega).
$$
By \cite[(5.1)--(5.2)]{ADH}, $P_t$ admits a symmetric non-negative kernel:
$$
        P_tf(x)=\int_{\mathbb R^N}p_t(x,y)f(y)\,d\omega(y),   \qquad     p_t(x,y)=p_t(y,x)\ge0.
$$
Moreover,
$$
        \int_{\mathbb R^N}p_t(x,y)\,d\omega(y)=P_t1(x)=1.
$$
For every non-negative Borel function $h$, the same integral defines $P_th(x)\in[0,\infty]$.

Let $f\in C_c^\infty(\mathbb R^N)$ and $t>0$.  The multiplier formula yields
$$
    \bigl\|\|\xi\|^2\mathcal F_\kappa(P_tf)(\xi)\bigr\|_2
    \le \sup_{r\ge0}r^2e^{-tr}\,\|f\|_2<\infty,
$$
so $P_tf\in D(-\Delta_\kappa)$.  Moreover,
\begin{align*}
    \mathcal F_\kappa(\partial_t^2P_tf)(\xi)  &=\|\xi\|^2e^{-t\|\xi\|}\mathcal F_\kappa f(\xi)
     =\mathcal F_\kappa(-\Delta_\kappa P_tf)(\xi).
\end{align*}
Thus $\partial_t^2P_tf=-\Delta_\kappa P_tf$ in $L^2(d\omega)$.  By
\cite[Proposition~5.1(c)]{ADH}, $P_tf$ is smooth.  For $\phi\in C_c^\infty(\mathbb R^N)$,
self-adjointness and the Dunkl integration-by-parts formula
\cite[Proposition~2.1, pp.~101--102]{RM} give
\begin{align*}
    \langle-\Delta_\kappa P_tf,\phi\rangle   &=\langle P_tf,-\Delta_\kappa\phi\rangle    =\langle-\Delta_\kappa P_tf,\phi\rangle,
\end{align*}
where the first and last occurrences of $-\Delta_\kappa P_tf$ denote, respectively, the self-adjoint realization
and the smooth differential-difference expression.  Hence
$$
        \partial_t^2P_tf=-\Delta_\kappa P_tf     \qquad\text{in }L^2(d\omega).
$$
Both sides are continuous.  Since $d\omega$ has full support,
$$
        \Delta_{\kappa,t}P_tf(x)=0,   \qquad x\in\mathbb R^N,\quad t>0.
$$

The Poisson semigroup is $G$-equivariant.

\begin{lemma}\label{lem:peq}
For every bounded Borel function $f$, every $\sigma\in G$, and every $t>0$,
$$
        P_t(f\circ\sigma)=(P_tf)\circ\sigma.
$$
Equivalently,
$$
        p_t(\sigma x,\sigma y)=p_t(x,y),
        \qquad x,y\in\mathbb R^N.
$$
Consequently, $P_tf$ is $G$-invariant whenever $f$ is $G$-invariant.
\end{lemma}

\begin{proof}
Let $U_\sigma f=f\circ\sigma$.  The $G$-invariance of $d\omega$ and the covariance of the Dunkl kernel imply
$$
        \mathcal F_\kappa(U_\sigma f)(\xi)
        =\mathcal F_\kappa f(\sigma\xi).
$$
Since $\|\sigma\xi\|=\|\xi\|$, for every $f\in L^2(d\omega)$,
\begin{align*}
     \mathcal F_\kappa(P_tU_\sigma f)(\xi)  &=e^{-t\|\xi\|}\mathcal F_\kappa f(\sigma\xi)
      =e^{-t\|\sigma\xi\|}\mathcal F_\kappa f(\sigma\xi)
      =\mathcal F_\kappa(U_\sigma P_tf)(\xi).
\end{align*}
Hence $P_tU_\sigma=U_\sigma P_t$ on $L^2(d\omega)$.

For $f\in C_c^\infty(\mathbb R^N)$, the kernel representation gives
\begin{align*}
     0=P_t(U_\sigma f)(x)-U_\sigma(P_tf)(x)   =\int_{\mathbb R^N}
       \bigl[p_t(x,\sigma^{-1}z)-p_t(\sigma x,z)\bigr]f(z)\,d\omega(z).
\end{align*}
The bracket is continuous in $z$ and integrates to zero against every
$f\in C_c^\infty(\mathbb R^N)$.  Since $d\omega$ has full support,
$$
        p_t(x,\sigma^{-1}z)=p_t(\sigma x,z),    \qquad x,z\in\mathbb R^N.
$$
Taking $z=\sigma y$ yields
$$
        p_t(\sigma x,\sigma y)=p_t(x,y).
$$
Therefore, for every bounded Borel function $f$,
\begin{align*}
     P_t(f\circ\sigma)(x)   &=\int_{\mathbb R^N}p_t(x,\sigma^{-1}z)f(z)\,d\omega(z)
     =\int_{\mathbb R^N}p_t(\sigma x,z)f(z)\,d\omega(z)    =(P_tf)(\sigma x).
\end{align*}
\end{proof}

By \cite[Proposition~5.1(a)]{ADH}, for every $t>0$ and $x,y\in\mathbb R^N$,
\begin{equation}\label{eq:psize}
        0\le p_t(x,y)      \le \frac{C}{V(x,y,t+d(x,y))}\,\frac{t}{t+d(x,y)}.
\end{equation}

\begin{lemma}\label{lem:lp}
For every $h\in L^2(\mathbb R^N,d\omega)$,
\begin{equation}\label{eq:lp2}
        \int_0^\infty\int_{\mathbb R^N}    \Gamma_{\kappa,t}(P_th)(y,t)\,t\,d\omega(y)\,dt     =\frac12\|h\|_{L^2(d\omega)}^2.
\end{equation}
Consequently,
\begin{equation}\label{eq:lpg}
        \int_0^\infty\int_{\mathbb R^N}   |\nabla_{y,t}P_th(y)|^2\,t\,d\omega(y)\,dt    \le\frac12\|h\|_{L^2(d\omega)}^2.
\end{equation}
\end{lemma}

\begin{proof}
Let $E$ be the spectral resolution of $-\Delta_\kappa$, and let
$$
        \mu_h(B):=\langle E(B)h,h\rangle.
$$
Fix $s>0$ and write $F=P_sh$.  Then
\begin{align*}
     \|-\Delta_\kappa F\|_2^2  &=\int_{[0,\infty)}\lambda^2e^{-2s\sqrt\lambda}\,d\mu_h(\lambda)
     \le\sup_{\lambda\ge0}\lambda^2e^{-2s\sqrt\lambda}\,\|h\|_2^2    <\infty,
\end{align*}
so $F\in D(-\Delta_\kappa)$.

For every non-negative integer $m$ and multi-index $\nu$,
\cite[Proposition~5.1(c)]{ADH}, symmetry, and the semigroup identity imply
\begin{align*}
     \|\partial_s^m\partial_y^\nu p_s(y,\cdot)\|_2^2   &\le Cs^{-2m-2|\nu|}
       \int_{\mathbb R^N}p_s(y,z)^2\,d\omega(z)\\
     &=Cs^{-2m-2|\nu|}p_{2s}(y,y)<\infty.
\end{align*}
The case $m=|\nu|=0$ and $L^2$ density extend the kernel formula for $P_s$ to
$h\in L^2(d\omega)$.  Applying the higher-order bounds also to one additional derivative shows that
$(s,y)\mapsto p_s(y,\cdot)$ is a smooth $L^2(d\omega)$-valued map.  Thus the kernel formula gives a smooth
representative of $F$.

For $\phi\in C_c^\infty(\mathbb R^N)$, self-adjointness and the Dunkl integration-by-parts formula give
\begin{align*}
     \langle-\Delta_\kappa F,\phi\rangle   &=\langle F,-\Delta_\kappa\phi\rangle    =\langle-\Delta_\kappa F,\phi\rangle,
\end{align*}
where the first and last terms denote the self-adjoint realization and the smooth differential-difference expression,
respectively.

Choose a non-increasing $\eta\in C_c^\infty([0,\infty))$ such that
$0\le\eta\le1$, $\eta=1$ on $[0,1]$, and $\eta=0$ on $[4,\infty)$.
Let $\eta_R(y):=\eta(\|y\|^2/R^2)$.  Since $\eta_R$ is radial,
$$
     \Delta_\kappa\eta_R(y)   =\frac{4\|y\|^2}{R^4}\eta''(\|y\|^2/R^2)   +\frac{2\mathbf N}{R^2}\eta'(\|y\|^2/R^2),
$$
and therefore
\begin{equation}\label{eq:cut}
        |\Delta_\kappa\eta_R(y)|   \le CR^{-2}\one_{\{R\le\|y\|\le2R\}}(y).
\end{equation}
Choose a radial $\chi_R\in C_c^\infty(\mathbb R^N)$ such that
$\chi_R=1$ on $B(0,4R)$ and
$\operatorname{supp}\chi_R\subset B(0,8R)$.  Since reflections preserve $\|y\|$,
\begin{align*}
     \int_{\mathbb R^N}\eta_R\Delta_\kappa|F|^2\,d\omega
     &=\int_{\mathbb R^N}\eta_R\Delta_\kappa(\chi_R|F|^2)\,d\omega
     =\int_{\mathbb R^N}|F|^2\Delta_\kappa\eta_R\,d\omega.
\end{align*}
Consequently, \eqref{eq:gprod} gives
\begin{align*}
     2\int_{\mathbb R^N}\eta_R\Gamma_\kappa(F)\,d\omega
     &=\int_{\mathbb R^N}|F|^2\Delta_\kappa\eta_R\,d\omega
       +2\Real\int_{\mathbb R^N}\eta_R(-\Delta_\kappa F)\,\overline F\,d\omega
     \longrightarrow   2\langle-\Delta_\kappa F,F\rangle.
\end{align*}
Indeed, \eqref{eq:cut} bounds the first term by
$CR^{-2}\|F\|_2^2$, while
$(-\Delta_\kappa F)\overline F\in L^1(d\omega)$.
Since $0\le\eta_R\uparrow1$, monotone convergence yields
\begin{equation}\label{eq:form}
        \int_{\mathbb R^N}\Gamma_\kappa(P_sh)\,d\omega   =\langle-\Delta_\kappa P_sh,P_sh\rangle
        =\|\sqrt{-\Delta_\kappa}\,P_sh\|_2^2.
\end{equation}

Finally,
$\partial_tP_th=-\sqrt{-\Delta_\kappa}\,P_th$,
\eqref{eq:form}, $\mu_h(\{0\})=0$, and Tonelli's theorem give
\begin{align*}
     \int_0^\infty\int_{\mathbb R^N}   \Gamma_{\kappa,t}(P_th)(y,t)\,t\,d\omega(y)\,dt
     &=2\int_{(0,\infty)}    \left(\int_0^\infty\lambda e^{-2t\sqrt\lambda}t\,dt\right)d\mu_h(\lambda)\\
     &=\frac12\mu_h((0,\infty))   =\frac12\|h\|_{L^2(d\omega)}^2.
\end{align*}
The pointwise bound
$|\nabla_{y,t}P_th|^2\le\Gamma_{\kappa,t}(P_th)$ now gives
\begin{align*}
     \int_0^\infty\int_{\mathbb R^N}  |\nabla_{y,t}P_th(y)|^2\,t\,d\omega(y)\,dt
     &\le\int_0^\infty\int_{\mathbb R^N}   \Gamma_{\kappa,t}(P_th)(y,t)\,t\,d\omega(y)\,dt
     =\frac12\|h\|_{L^2(d\omega)}^2.
\end{align*}
\end{proof}

\subsection{Orbit maximal operators}

For $h\in L^1_{\mathrm{loc}}(\mathbb R^N,d\omega)$, define
\begin{equation*}
        \mathcal{M}_{\calO}h(x)     :=\sup_{r>0}\frac{1}{\omega(\calO(x,r))}
        \int_{\calO(x,r)}|h(y)|\,d\omega(y).
\end{equation*}
Let $\mathcal M_{\mathrm{HL}}$ be the Hardy--Littlewood maximal operator associated with Euclidean balls and $d\omega$.
For every $r>0$, \eqref{eq:ovol} and the $G$-invariance of $d\omega$ imply
\begin{align*}
     \frac{1}{\omega(\calO(x,r))}\int_{\calO(x,r)}|h(y)|\,d\omega(y)
     &\le\sum_{\sigma\in G}\frac{1}{V(\sigma x,r)}
       \int_{B(\sigma x,r)}|h(y)|\,d\omega(y)  \le\sum_{\sigma\in G}\mathcal M_{\mathrm{HL}}h(\sigma x).
\end{align*}
Thus
$$
        \mathcal M_{\calO}h(x)
        \le\sum_{\sigma\in G}\mathcal M_{\mathrm{HL}}h(\sigma x).
$$
Since $d\omega$ is doubling and $G$ is finite, $\mathcal M_{\calO}$ is of weak type $(1,1)$ and bounded on $L^p(d\omega)$
for $1<p\le\infty$.

\begin{lemma}\label{lem:pmax}
There exists a structural constant $C_P>0$ such that, for every non-negative
$h\in L^1_{\mathrm{loc}}(\mathbb R^N,d\omega)$, $x,y\in\mathbb R^N$, and $t>0$ with $d(x,y)<t$,
$$
        P_th(y)\le C_P\mathcal M_{\calO}h(x).
$$
\end{lemma}

\begin{proof}
Decompose $\mathbb R^N$ into the orbit annuli
$$
        A_0=\calO(y,2t),   \qquad
        A_k=\calO(y,2^{k+1}t)\setminus\calO(y,2^kt),    \quad k\ge1.
$$
For $z\in A_k$, $k\ge0$, we have $t+d(y,z)\simeq2^kt$.  Hence
$$
        \frac{t}{t+d(y,z)}\lesssim\frac{t}{2^kt}=2^{-k},
$$
while \eqref{eq:ovol} and doubling imply
\begin{align*}
     \omega(\calO(y,2^{k+1}t))   &\le |G|V(y,2^{k+1}t)   \lesssim V(y,2^kt)
     \le V(y,z,t+d(y,z)).
\end{align*}
Combining these estimates with \eqref{eq:psize}, we obtain
$$
        p_t(y,z)   \lesssim\frac{2^{-k}}{\omega(\calO(y,2^{k+1}t))},
        \qquad z\in A_k.
$$

Since $d(x,y)<t$, the triangle inequality implies
$$
     \calO(y,2^{k+1}t)   \subset\calO(x,(2^{k+1}+1)t)   \subset\calO(x,3\cdot2^kt)
     \subset\calO(y,(3\cdot2^k+1)t)    \subset\calO(y,2^{k+2}t).
$$
The doubling property of orbit balls therefore shows that
$$
     \omega(\calO(y,2^{k+1}t))   \le\omega(\calO(x,3\cdot2^kt))    \lesssim\omega(\calO(y,2^{k+1}t)).
$$
Consequently,
\begin{align*}
     P_th(y)   &=\sum_{k=0}^\infty\int_{A_k}p_t(y,z)h(z)\,d\omega(z)\\
     &\lesssim\sum_{k=0}^\infty   \frac{2^{-k}}{\omega(\calO(x,3\cdot2^kt))}
       \int_{\calO(x,3\cdot2^kt)}h(z)\,d\omega(z)\\
     &\le C\mathcal M_{\calO}h(x)\sum_{k=0}^\infty2^{-k}   \le C_P\mathcal M_{\calO}h(x).
\end{align*}
\end{proof}

\begin{corollary}\label{cor:nmax}
Fix $\beta>0$.  Let $f\in L^1_{\mathrm{loc}}(\mathbb R^N,d\omega)$ and assume that
$P_t|f|(y)<\infty$ for every $y\in\mathbb R^N$ and $t>0$.  Then there exists $C_\beta>0$, depending only on $\beta$
and the structural constants, such that
\begin{equation}\label{eq:nmax}
        \mathcal N_P^\beta f(x)\le C_\beta\mathcal M_{\calO}f(x),    \qquad x\in\mathbb R^N.
\end{equation}
The hypothesis holds for $f\in C_c^\infty(\mathbb R^N)$.  It also holds for $f=\one_F$ with $F$ measurable, since
$0\le P_t\one_F\le P_t1=1$.
\end{corollary}

\begin{proof}
Fix $x\in\mathbb R^N$ and $(y,t)$ with $d(x,y)<\beta t$.  Let
$$
        A_0=\calO(y,2t),   \qquad
        A_k=\calO(y,2^{k+1}t)\setminus\calO(y,2^kt),   \quad k\ge1.
$$
The annular kernel estimate from Lemma~\ref{lem:pmax} remains valid:
$$
        p_t(y,z)   \lesssim\frac{2^{-k}}{\omega(\calO(y,2^{k+1}t))},
        \qquad z\in A_k,\quad k\ge0.
$$
Let $r_k=2^{k+1}t$.  The triangle inequality implies
$$
        \calO(y,r_k)    \subset\calO(x,r_k+\beta t)    \subset\calO(y,r_k+2\beta t).
$$
Since $(r_k+2\beta t)/r_k=1+\beta/2^k\le1+\beta$, orbit doubling shows that
$$
        \omega(\calO(y,r_k))   \le\omega(\calO(x,r_k+\beta t))
        \le\omega(\calO(y,r_k+2\beta t))    \lesssim_\beta\omega(\calO(y,r_k)).
$$
By positivity,
\begin{align*}
     |P_tf(y)|   &\le P_t|f|(y)   \lesssim\sum_{k=0}^\infty
       \frac{2^{-k}}{\omega(\calO(x,r_k+\beta t))}
       \int_{\calO(x,r_k+\beta t)}|f(z)|\,d\omega(z)\\
     &\le C_\beta\mathcal M_{\calO}f(x)\sum_{k=0}^\infty2^{-k}   \le C_\beta\mathcal M_{\calO}f(x).
\end{align*}
Taking the supremum over $d(x,y)<\beta t$ proves \eqref{eq:nmax}.
\end{proof}

\begin{lemma}\label{lem:ptail}
There exists a structural constant $C_{\mathrm{tail}}>0$ such that, for every measurable set
$E\subset\mathbb R^N$, $y\in\mathbb R^N$, $t>0$, and $\beta\ge1$,
$$
        d(y,E):=\inf_{z\in E}d(y,z)\ge\beta t    \quad\Longrightarrow\quad     P_t\one_E(y)\le C_{\mathrm{tail}}\beta^{-1}.
$$
\end{lemma}

\begin{proof}
Let
$$
        A_k:=\{z:2^k\beta t\le d(y,z)<2^{k+1}\beta t\},
        \qquad k\ge0.
$$
Then $E\subset\bigcup_{k\ge0}A_k$.  For $z\in A_k$, the condition $\beta\ge1$ implies
$t+d(y,z)\simeq2^k\beta t$.  Consequently,
$$
        \frac{t}{t+d(y,z)}\lesssim\frac{t}{2^k\beta t}=(2^k\beta)^{-1}.
$$
At the scale $2^k\beta t$, the volume comparison used in Lemma~\ref{lem:pmax} shows
\begin{align*}
     \omega(\calO(y,2^{k+1}\beta t))  &\le |G|V(y,2^{k+1}\beta t)
     \lesssim V(y,2^k\beta t)   \le V(y,z,t+d(y,z)).
\end{align*}
Combining the preceding volume bound with \eqref{eq:psize}, we obtain
\begin{align*}
     P_t\one_E(y)   &=\sum_{k=0}^\infty\int_{E\cap A_k}p_t(y,z)\,d\omega(z)
     \le C\sum_{k=0}^\infty(2^k\beta)^{-1}  \frac{\omega(E\cap A_k)}{\omega(\calO(y,2^{k+1}\beta t))}\\
     &\le C\beta^{-1}\sum_{k=0}^\infty2^{-k}    \le C_{\mathrm{tail}}\beta^{-1}.
\end{align*}
\end{proof}

\section{Invariant localization over the good set}\label{sec:loc}

Fix $f\in C_c^\infty(\mathbb R^N)$ and $\lambda>0$.  We construct a $G$-invariant cut-off that equals one on the orbit
tent over a thick subset of the good set and vanishes outside a larger tent over the good set.  This converts the
conical energy into a localized integral on the upper half-space.

\subsection{The full-space orbit objects}

For $\beta>0$, define the orbit non-tangential maximal function by
$$
        \mathcal N_P^\beta f(x)    :=\sup_{\substack{y\in\mathbb R^N,\,t>0\\ d(x,y)<\beta t}}|P_tf(y)|.
$$
The intrinsic Dunkl area function is
\begin{equation}\label{eq:area}
        \mathcal S_Pf(x)   :=\left(   \int_0^\infty\int_{d(x,y)<t}    \Gamma_{\kappa,t}(P_tf)(y,t)\,
        \frac{t\,d\omega(y)\,dt}{V(x,t)}    \right)^{1/2}.
\end{equation}
Its integrand is the full space-time Dunkl energy:
\begin{align*}
     \Gamma_{\kappa,t}(P_tf)(y,t)   &=|\partial_tP_tf(y)|^2+|\nabla_yP_tf(y)|^2
      +\sum_{\alpha\in R_+}\kappa(\alpha)  \frac{|P_tf(y)-P_tf(\sigma_\alpha y)|^2}
     {\langle\alpha,y\rangle^2}.
\end{align*}
The last sum is the reflection energy and must be retained for general $f$.

\subsection{The invariant good set and its Poisson extension}

Let $u(y,t):=P_tf(y)$.  Let $\beta>1$, and define
$$
        E=E_\beta(\lambda)   :=\{x\in\mathbb R^N:\mathcal N_P^\beta f(x)\le\lambda\}.
$$
All objects built from $E_\beta(\lambda)$ below depend on $\beta$, which will be fixed after Lemma~\ref{lem:vlo}.
For every $\sigma\in G$, the $G$-invariance of the orbit distance implies
\begin{align*}
     \mathcal N_P^\beta f(\sigma x)
     &=\sup_{\substack{y\in\mathbb R^N,\,t>0\\ d(\sigma x,y)<\beta t}}
       |P_tf(y)|  =\sup_{\substack{y\in\mathbb R^N,\,t>0\\ d(x,y)<\beta t}}
       |P_tf(y)|   =\mathcal N_P^\beta f(x).
\end{align*}
Hence $E$ is $G$-invariant.  Moreover, for every $a\in\mathbb R$,
$$
        \left\{x:\mathcal N_P^\beta f(x)>a\right\}
        =\bigcup_{\substack{y\in\mathbb R^N,\,t>0\\ |P_tf(y)|>a}}    \calO(y,\beta t),
$$
which is open.  Thus $\mathcal N_P^\beta f$ is lower semicontinuous, $E$ is closed, and $E^c$ is open.

By Corollary~\ref{cor:nmax} and the weak type $(1,1)$ estimate for $\mathcal M_{\calO}$, we have
\begin{align}\label{eq:bad}
     \|\one_{E^c}\|_{L^2(d\omega)}^2  =\omega(E^c)
     &=\omega\{x:\mathcal N_P^\beta f(x)>\lambda\}
     \le\omega\{x:C_\beta\mathcal M_{\calO}f(x)>\lambda\}
     \le\frac{C_\beta}{\lambda}\|f\|_{L^1(d\omega)}<\infty.
\end{align}

Let
$$
        g:=\one_E,   \qquad    v(y,t):=P_tg(y).
$$
Since $E$ is Borel and $G$-invariant, $g$ is a bounded, $G$-invariant Borel function.  Lemma~\ref{lem:peq} implies
$$
        v(\sigma y,t)   =(P_tg)(\sigma y)   =P_t(g\circ\sigma)(y)
        =P_tg(y)   =v(y,t),     \qquad \sigma\in G.
$$
By positivity and conservation, we have
\begin{equation}\label{eq:vbad}
        0\le v(y,t)   =P_t\one_E(y)    =1-P_t\one_{E^c}(y)  \le1.
\end{equation}
Every scalar function of $v$ is therefore $G$-invariant in $y$.
By \eqref{eq:bad}, $\one_{E^c}\in L^2(d\omega)$; hence \eqref{eq:lp2} and \eqref{eq:lpg} apply to
$P_t\one_{E^c}$.

\subsection{The auxiliary thick good set}

Enlarge the structural constant $C_P$ in Lemma~\ref{lem:pmax}, if necessary, so that $C_P\ge1$.  Define
\begin{equation}\label{eq:aset}
        A=A_\beta(\lambda)    :=\left\{x\in\mathbb R^N:
        \mathcal M_{\calO}(\one_{E^c})(x)\le\frac1{20C_P}\right\}.
\end{equation}
Since $\calO(\sigma x,r)=\calO(x,r)$, the set $A$ is $G$-invariant.  By \eqref{eq:bad} and the weak type $(1,1)$
estimate for $\mathcal M_{\calO}$, we have
\begin{align*}
     \omega(A^c)  &=\omega\left\{x:\mathcal M_{\calO}(\one_{E^c})(x)>\frac1{20C_P}\right\}
     \le C\|\one_{E^c}\|_{L^1(d\omega)}   =C\omega(E^c).
\end{align*}

Moreover, $E^c\subset A^c$.  Indeed, if $x\in E^c$, the openness and $G$-invariance of $E^c$ ensure that some $r>0$
satisfies
$$
        B(\sigma x,r)=\sigma B(x,r)\subset E^c,   \qquad \sigma\in G.
$$
Consequently, $\calO(x,r)=\bigcup_{\sigma\in G}B(\sigma x,r)\subset E^c$, and
\begin{align*}
     \mathcal M_{\calO}(\one_{E^c})(x)  &\ge\frac1{\omega(\calO(x,r))}
       \int_{\calO(x,r)}\one_{E^c}\,d\omega  =1>\frac1{20C_P},
\end{align*}
so $x\in A^c$.  Thus $A\subset E$, and Chebyshev's inequality shows that
\begin{align}\label{eq:split}
     \omega\{x:\mathcal S_Pf(x)>\lambda\}
     &\le\omega(A^c)+\omega\{x\in A:\mathcal S_Pf(x)>\lambda\}\notag\\
     &\le C\omega(E^c)+\frac1{\lambda^2}\int_A\mathcal S_Pf(x)^2\,d\omega(x).
\end{align}
Thus the reduction is complete once the last integral is estimated.

\subsection{The two regions in the upper half-space}

Define the orbit tent
$$
        W:=\bigcup_{x\in A}\{(y,t)\in\mathbb R^{N+1}_+:d(x,y)<t\}.
$$
The enlarged tent is
$$
        \widetilde W:=\bigcup_{x\in E}\{(y,t)\in\mathbb R^{N+1}_+:d(x,y)<\beta t\}.
$$
Since $A\subset E$ and $\beta>1$, we have $W\subset\widetilde W$.  For $t>0$, let
$$
        W_t:=\{y\in\mathbb R^N:(y,t)\in W\}.
$$

\begin{lemma}\label{lem:vhi}
For every $(y,t)\in W$, we have
$$
        v(y,t)\ge\frac{19}{20}.
$$
\end{lemma}

\begin{proof}
Choose $x\in A$ with $d(x,y)<t$.  Then Lemma~\ref{lem:pmax}, \eqref{eq:aset}, and \eqref{eq:vbad} imply
\begin{align*}
     v(y,t)   &=1-P_t\one_{E^c}(y)   \ge1-C_P\mathcal M_{\calO}(\one_{E^c})(x)
     \ge\frac{19}{20}.
\end{align*}
\end{proof}

Let $C_{\mathrm{tail}}$ be the constant in Lemma~\ref{lem:ptail}, and let
$$
        c_2:=\frac1{40},   \qquad    \beta_0:=\max\left\{2,\frac{2C_{\mathrm{tail}}}{c_2}\right\}.
$$

\begin{lemma}\label{lem:vlo}
Whenever $\beta\ge\beta_0$, the objects corresponding to this value of $\beta$ satisfy
$$
        (y,t)\notin\widetilde W   \quad\Longrightarrow\quad     v(y,t)<c_2.
$$
\end{lemma}

\begin{proof}
If $(y,t)\notin\widetilde W$, then $d(y,E)\ge\beta t$.  Since $\beta\ge\beta_0$, Lemma~\ref{lem:ptail} ensures that
\begin{align*}
     v(y,t)   &=P_t\one_E(y)   \le\frac{C_{\mathrm{tail}}}{\beta}
     \le\frac{c_2}{2}<c_2.
\end{align*}
\end{proof}

Fix one $\beta\ge\beta_0$ and henceforth use the corresponding sets $E$, $A$, $W$, and $\widetilde W$ and functions $g$
and $v$.  Choose
$$
        c_2<c_1<c_3<\frac{19}{20}
$$
and a smooth cut-off $\varphi\in C^\infty(\mathbb R)$ such that
$$
        0\le\varphi\le1,    \qquad    \varphi(s)=0 \quad \text{for }s\le c_1,
        \qquad    \varphi(s)=1 \quad \text{for }s\ge c_3.
$$
Then $\varphi(0)=0$, $\varphi(1)=1$, and
$$
        \operatorname{supp}\varphi'\cup\operatorname{supp}\varphi''
        \subset[c_1,c_3]\Subset(c_2,19/20).
$$
Lemmas~\ref{lem:vhi} and~\ref{lem:vlo} imply
\begin{align*}
     (y,t)\in W   &\Longrightarrow v(y,t)\ge\frac{19}{20}>c_3   \Longrightarrow\varphi(v(y,t))=1,\\
     (y,t)\notin\widetilde W   &\Longrightarrow v(y,t)<c_2<c_1
     \Longrightarrow\varphi(v(y,t))=\varphi'(v(y,t))=\varphi''(v(y,t))=0.
\end{align*}
Consequently,
$$
     \{(y,t)\in\mathbb R^{N+1}_+:\varphi'(v(y,t))\ne0\text{ or }\varphi''(v(y,t))\ne0\}
     \subset\widetilde W\setminus W.
$$

\subsection{Reduction to a global energy integral}

If $x\in\calO(y,t)$, then $V(x,t)\simeq V(y,t)$, while \eqref{eq:ovol} implies
$\omega(\calO(y,t))\lesssim V(y,t)$.  The integrand in \eqref{eq:area} is non-negative, so Tonelli's theorem applies.
Since $\varphi(v)=1$ on $W$,
\begin{align}\label{eq:glob}
     \int_A\mathcal S_Pf(x)^2\,d\omega(x)  &=\int_0^\infty\int_{W_t}\Gamma_{\kappa,t}(u)(y,t)\,t
       \left(\int_{A\cap\calO(y,t)}\frac{d\omega(x)}{V(x,t)}\right)d\omega(y)\,dt\\
     &\le C\int_0^\infty\int_{W_t}\Gamma_{\kappa,t}(u)(y,t)\,t
       \frac{\omega(\calO(y,t))}{V(y,t)}\,d\omega(y)\,dt\notag\\
     &\le C\int_0^\infty\int_{W_t}\Gamma_{\kappa,t}(u)(y,t)\,t\,d\omega(y)\,dt\notag\\
     &=C\int_0^\infty\int_{W_t}  \Gamma_{\kappa,t}(u)(y,t)\varphi(v(y,t))^2\,t\,d\omega(y)\,dt\notag\\
     &\le C\int_0^\infty\int_{\mathbb R^N}  \Gamma_{\kappa,t}(u)(y,t)\varphi(v(y,t))^2\,t\,d\omega(y)\,dt.\notag
\end{align}
The localized energy estimate below controls the last integral.

\section{The localized energy estimate}\label{sec:eng}

Retain $u(y,t)=P_tf(y)$ and $v(y,t)=P_t\one_E(y)$, and let $a:=\varphi(v)$.
The $G$-invariance of $v$ makes $a$ $G$-invariant in $y$; hence the reflection part of
$\Gamma_{\kappa,t}(ua)$ is $a^2$ times that of $\Gamma_{\kappa,t}(u)$.

Choose $\Psi\in C_c^\infty((c_2,19/20))$ such that
$$
        0\le\Psi\le1,   \qquad     \Psi=1\quad\text{on a neighborhood of }[c_1,c_3].
$$
Since $\operatorname{supp}\varphi'\cup\operatorname{supp}\varphi''\subset[c_1,c_3]$ and $\Psi=1$ near
$[c_1,c_3]$, there exists $C_\varphi>0$ such that
\begin{equation}\label{eq:psi}
        |\varphi'(s)|^2+|\varphi(s)\varphi''(s)|    \le C_\varphi\Psi(s)^2,    \qquad s\in\mathbb R.
\end{equation}

\begin{lemma}
There exists a constant $C>0$, depending only on the underlying Dunkl structure and the fixed cut-offs $\varphi$ and
$\Psi$, but independent of $f$ and $\lambda$, such that
\begin{align}\label{eq:loc}
     &\int_0^\infty\int_{\mathbb R^N}   a(y,t)^2\Gamma_{\kappa,t}(u)(y,t)\,t\,d\omega(y)\,dt\\
     \le\,& C\int_E|f(y)|^2\,d\omega(y)   +C\int_0^\infty\int_{\mathbb R^N}
      |u(y,t)|^2\Psi(v(y,t))^2|\nabla_{y,t}v(y,t)|^2\,t\,d\omega(y)\,dt.\notag
\end{align}
\end{lemma}

\begin{proof}
We first prove a uniform estimate for smooth invariant data, remove the spatial truncation, and compute the temporal
traces.  We then pass to $\one_{E^c}$ by $L^2$ approximation and \eqref{eq:lpg}.

\medskip
\noindent\textit{Step 1: smooth invariant data.}
Let $b_0\in C_c^\infty(\mathbb R^N)$ be real-valued and $G$-invariant, with $0\le b_0\le1$, and let
$$
        v_0:=1-P_tb_0,    \qquad    a_0:=\varphi(v_0),   \qquad     h_0:=ua_0.
$$
We prove
\begin{align}\label{eq:loc0}
     &\int_0^\infty\int_{\mathbb R^N}    a_0^2\Gamma_{\kappa,t}(u)\,t\,d\omega\,dt\\
     \le\,& C\int_{\mathbb R^N}|f|^2\varphi(1-b_0)^2\,d\omega
      +C\int_0^\infty\int_{\mathbb R^N}  |u|^2\Psi(v_0)^2|\nabla_{y,t}v_0|^2\,t\,d\omega\,dt,\notag
\end{align}
with $C$ independent of $b_0$.  In Steps 1--3, write $b,v,a,h$ for $b_0,v_0,a_0,h_0$.
Positivity and conservation imply $0\le v\le1$.  Lemma~\ref{lem:peq} shows that
$v(\sigma y,t)=v(y,t)$; hence $a=\varphi(v)$ is real-valued and $G$-invariant.
Independently, Poisson harmonicity ensures that $\Delta_{\kappa,t}u=\Delta_{\kappa,t}v=0$.

For the reflection differences,
$$
        h(y,t)-h(\sigma_\alpha y,t)    =a(y,t)\bigl(u(y,t)-u(\sigma_\alpha y,t)\bigr).
$$
For the ordinary gradient, $a\nabla_{y,t}u=\nabla_{y,t}h-u\nabla_{y,t}a$.  Therefore
\begin{align*}
     a^2\Gamma_{\kappa,t}(u)   &=a^2|\nabla_{y,t}u|^2   +a^2\bigl(\Gamma_{\kappa,t}(u)-|\nabla_{y,t}u|^2\bigr)\\
     &\le2|\nabla_{y,t}h|^2+2|u|^2|\nabla_{y,t}a|^2   +a^2\bigl(\Gamma_{\kappa,t}(u)-|\nabla_{y,t}u|^2\bigr)\\
     &\le2\Gamma_{\kappa,t}(h)+2|u|^2|\nabla_{y,t}a|^2.
\end{align*}
The square identity \eqref{eq:gsq} shows
$$
     \Gamma_{\kappa,t}(h)  =\frac12\Delta_{\kappa,t}|h|^2   -\Real\bigl(\overline h\,\Delta_{\kappa,t}h\bigr).
$$
Since $\Delta_{\kappa,t}u=\Delta_{\kappa,t}v=0$ and $a=\varphi(v)$ is $G$-invariant,
\eqref{eq:prodt} and \eqref{eq:chain} imply
\begin{align*}
     \Delta_{\kappa,t}h  &=u\Delta_{\kappa,t}a+2\nabla_{y,t}u\cdot\nabla_{y,t}a\\
     &=u\varphi''(v)|\nabla_{y,t}v|^2    +2\varphi'(v)\nabla_{y,t}u\cdot\nabla_{y,t}v.
\end{align*}
For $\delta>0$, Young's inequality and \eqref{eq:psi} imply
\begin{align*}
     \big|\Real(\overline h\,\Delta_{\kappa,t}h)\big|   &\le2|ua||\nabla_{y,t}u||\varphi'(v)||\nabla_{y,t}v|
       +|u|^2|\varphi(v)\varphi''(v)||\nabla_{y,t}v|^2\\
     &\le\delta a^2\Gamma_{\kappa,t}(u)   +C_\delta|u|^2\Psi(v)^2|\nabla_{y,t}v|^2.
\end{align*}
The remaining gradient term is estimated separately:
\begin{align*}
     2|u|^2|\nabla_{y,t}a|^2   &=2|u|^2|\varphi'(v)|^2|\nabla_{y,t}v|^2  \le C|u|^2\Psi(v)^2|\nabla_{y,t}v|^2.
\end{align*}

Choose $\eta$ with the same properties as in the proof of Lemma~\ref{lem:lp}, and let
$\eta_R(y):=\eta(\|y\|^2/R^2)$.  Then $0\le\eta_R\uparrow1$,
$\operatorname{supp}\eta_R\subset\overline{B(0,2R)}$, and \eqref{eq:cut} remains valid.
For $0<\eps<T$ and $R>0$, define
$$
     I_{R,\eps,T}   :=\int_\eps^T\int_{\mathbb R^N}   \eta_Ra^2\Gamma_{\kappa,t}(u)\,t\,d\omega\,dt.
$$
The truncated Laplacian term is
$$
     J_{R,\eps,T}  :=\int_\eps^T\int_{\mathbb R^N}   \eta_R\Delta_{\kappa,t}|h|^2\,t\,d\omega\,dt.
$$
The error term is
$$
     K_{R,\eps,T}  :=\int_\eps^T\int_{\mathbb R^N}
       \eta_R|u|^2\Psi(v)^2|\nabla_{y,t}v|^2\,t\,d\omega\,dt.
$$
After multiplication by $\eta_Rt$ and integration over $\mathbb R^N\times(\eps,T)$,
\begin{align*}
     I_{R,\eps,T}  &\le J_{R,\eps,T}    +2\int_\eps^T\int_{\mathbb R^N}
        \eta_R\big|\Real(\overline h\,\Delta_{\kappa,t}h)\big|\,t\,d\omega\,dt\\
     &\quad   +2\int_\eps^T\int_{\mathbb R^N}    \eta_R|u|^2|\nabla_{y,t}a|^2\,t\,d\omega\,dt\\
     &\le J_{R,\eps,T}+2\delta I_{R,\eps,T}+C_\delta K_{R,\eps,T}.
\end{align*}
Fix $0<\delta\le1/4$.  After absorption,
\begin{align}\label{eq:tr}
     I_{R,\eps,T}   &\le2J_{R,\eps,T}+CK_{R,\eps,T}
     \le2|J_{R,\eps,T}|+CK_{R,\eps,T},
\end{align}
where $C$ is independent of $R,\eps,T$, and $b_0$.

\medskip
\noindent\textit{Step 2: removal of the spatial truncation.}
Fix $t\in(\eps,T)$ and write $F_t:=|h(\cdot,t)|^2$.  Choose a radial $G$-invariant
$\chi_R\in C_c^\infty(\mathbb R^N)$ such that
$$
        \chi_R=1\quad\text{on }B(0,4R),   \qquad  \operatorname{supp}\chi_R\subset B(0,8R).
$$
Since $\operatorname{supp}\eta_R\subset\overline{B(0,2R)}$ and $\|\sigma y\|=\|y\|$, the ordinary derivatives and
reflected values of $F_t$ and $\chi_RF_t$ agree on $\operatorname{supp}\eta_R$.  By \eqref{eq:cut},
$\chi_R=1$ on $\operatorname{supp}\Delta_\kappa\eta_R$.  Hence, by the symmetry of $\Delta_\kappa$ on
$C_c^\infty(\mathbb R^N)$, we obtain
\begin{align*}
     \int_{\mathbb R^N}\eta_R\Delta_\kappa F_t\,d\omega
     &=\int_{\mathbb R^N}\eta_R\Delta_\kappa(\chi_RF_t)\,d\omega
     =\int_{\mathbb R^N}\chi_RF_t\Delta_\kappa\eta_R\,d\omega
           =\int_{\mathbb R^N}F_t\Delta_\kappa\eta_R\,d\omega.
\end{align*}
Since $|h|\le|u|$ and $P_t$ is an $L^2$ contraction, we obtain
\begin{align*}
     \left|\int_\eps^T\int_{\mathbb R^N}  |h|^2\Delta_\kappa\eta_R\,t\,d\omega\,dt\right|
     &\le\frac{C}{R^2}\int_\eps^Tt\|u(\cdot,t)\|_2^2\,dt
     \le\frac{C(T^2-\eps^2)}{2R^2}\|f\|_2^2  \longrightarrow0  \quad \text{as}\     R\to\infty.
\end{align*}

\medskip
\noindent\textit{Step 3: the temporal traces.}
Let
$$
        H_R(t):=\int_{\mathbb R^N}\eta_R|h(\cdot,t)|^2\,d\omega,    \qquad
        H(t):=\|h(\cdot,t)\|_2^2.
$$
For fixed $R$, kernel regularity and the compact support of $\eta_R$ permit two differentiations under the integral;
thus $H_R\in C^2([\eps,T])$ and
$$
     J_{R,\eps,T}  =\int_\eps^T\int_{\mathbb R^N}F_t\Delta_\kappa\eta_R\,t\,d\omega\,dt
      +[tH_R'(t)]_\eps^T-[H_R(t)]_\eps^T.
$$
For $q\in L^2(d\omega)$, the spectral calculus gives
$\partial_tP_tq=-\sqrt L e^{-t\sqrt L}q$.
The corresponding multiplier bound is
$$
        \|\partial_tP_tq\|_2   \le\sup_{\lambda\ge0}\sqrt\lambda e^{-t\sqrt\lambda}\,\|q\|_2
        =\frac1{et}\|q\|_2.
$$
Both $t\mapsto P_tq$ and $t\mapsto\partial_tP_tq$ are continuous from $[\eps,T]$ to $L^2(d\omega)$.  For either
$$
        (c_t,w_t)=\bigl(a(\cdot,t),\partial_tP_tf\bigr)    \quad\text{or}\quad
        (c_t,w_t)=\bigl(u(\cdot,t)\varphi'(v(\cdot,t)),\partial_tP_tb\bigr),
$$
the functions $c_t$ are pointwise continuous and uniformly bounded on $\mathbb R^N\times[\eps,T]$.  Thus, as
$t\to s$,
\begin{align*}
     \|c_tw_t-c_sw_s\|_2   &\le\|c_t\|_\infty\|w_t-w_s\|_2   +\|(c_t-c_s)w_s\|_2
     \longrightarrow0,
\end{align*}
where dominated convergence handles the second term.  Pointwise differentiation followed by Bochner integration
therefore gives, in $L^2(d\omega)$,
\begin{align*}
     h(\cdot,t)-h(\cdot,s)   &=\int_s^t\bigl[  a(\cdot,r)\partial_rP_rf
      -u(\cdot,r)\varphi'(v(\cdot,r))\partial_rP_rb  \bigr]\,dr.
\end{align*}
Consequently,
$$
        h\in C^1([\eps,T];L^2(d\omega)),   \qquad
        \partial_th=a\partial_tP_tf-u\varphi'(v)\partial_tP_tb.
$$
For $s=\eps,T$, monotone convergence for $H_R(s)$ and dominated convergence for $H_R'(s)$ show that
$H_R(s)\longrightarrow H(s)$ and
\begin{align*}
     H_R'(s)  =2\Real\int_{\mathbb R^N}\eta_R\overline{h(\cdot,s)}\,\partial_sh(\cdot,s)\,d\omega
     \longrightarrow2\Real\int_{\mathbb R^N}   \overline{h(\cdot,s)}\,\partial_sh(\cdot,s)\,d\omega
      =H'(s).
\end{align*}
Together with Step 2, these limits imply
\begin{equation}\label{eq:ttr}
     \lim_{R\to\infty}J_{R,\eps,T}   =[tH'(t)]_\eps^T-[H(t)]_\eps^T.
\end{equation}

Let $\mu_q$ be the spectral measure of $q\in L^2(d\omega)$ for $L$.  Since $r^2e^{-2r}$ is bounded on
$[0,\infty)$ and tends to zero as $r\to0$, dominated convergence shows that
$$
     \|t\partial_tP_tq\|_2^2   =\int_{[0,\infty)}(t\sqrt\lambda)^2e^{-2t\sqrt\lambda}\,d\mu_q(\lambda)
     \longrightarrow0   \qquad\text{as }t\to 0.
$$
Applying this to $q=f$ and $q=b$, and using $\|h(\cdot,t)\|_2\le\|f\|_2$ and
$\|u(\cdot,t)\|_\infty\le\|f\|_\infty$, we obtain, as $t\to 0$,
$$
     t\|\partial_th(\cdot,t)\|_2   \le t\|\partial_tP_tf\|_2
       +\|f\|_\infty\|\varphi'\|_\infty t\|\partial_tP_tb\|_2   \longrightarrow0.
$$
Therefore,
$$
     t|H'(t)|   \le2\|h(\cdot,t)\|_2t\|\partial_th(\cdot,t)\|_2   \longrightarrow0.
$$
At infinity,
$$
        0\le e^{-2t\sqrt\lambda}\le1,   \qquad
        e^{-2t\sqrt\lambda}\longrightarrow\one_{\{0\}}(\lambda).
$$
Since $\ker_{L^2(d\omega)}L=\{0\}$, dominated convergence shows that
$$
        \|P_tf\|_2^2   =\int_{[0,\infty)}e^{-2t\sqrt\lambda}\,d\mu_f(\lambda)
        \longrightarrow\mu_f(\{0\})=0.
$$
Moreover, the multiplier estimate shows that
$$
        \sup_{t>0}t\|\partial_tP_tq\|_2\le e^{-1}\|q\|_2,   \qquad q\in L^2(d\omega).
$$
Therefore, $H(t)\le\|P_tf\|_2^2\to0$ as $t\to\infty$.
The same multiplier bound also shows that
$$
     t\|\partial_th(\cdot,t)\|_2  \le t\|\partial_tP_tf\|_2
       +\|f\|_\infty\|\varphi'\|_\infty t\|\partial_tP_tb\|_2   \le C_{f,b}.
$$
Consequently,
$$
     t|H'(t)|  \le2\|h(\cdot,t)\|_2t\|\partial_th(\cdot,t)\|_2   \longrightarrow0.
$$

To compute the lower trace, strong continuity of $P_t$ on $L^2(d\omega)$ gives
$\|v(\cdot,t)-(1-b)\|_2=\|P_tb-b\|_2\to0$.  The Lipschitz continuity of $\varphi$ then implies
$\|a(\cdot,t)-\varphi(1-b)\|_2\le\operatorname{Lip}(\varphi)\|P_tb-b\|_2\to0$.
Combining this convergence with $P_tf\to f$ in $L^2(d\omega)$, we obtain
$$
     \|h(\cdot,t)-f\varphi(1-b)\|_2  \le\|P_tf-f\|_2
       +\|f\|_\infty\operatorname{Lip}(\varphi)\|P_tb-b\|_2  \longrightarrow0.
$$
Together with the estimates at infinity, this identifies all temporal traces:
$$
        H(0+)=\|f\varphi(1-b)\|_2^2,  \qquad  H(\infty)=0,
        \qquad   \lim_{t\to 0}tH'(t)=\lim_{t\to\infty}tH'(t)=0.
$$
For $n\ge2$, let $\eps_n=n^{-1}$ and $T_n=n$.  Equation~\eqref{eq:ttr} becomes
\begin{equation}\label{eq:bdry}
        \lim_{n\to\infty}\lim_{R\to\infty}J_{R,\eps_n,T_n}  =\|f\varphi(1-b)\|_2^2.
\end{equation}
Since the absolute value is continuous and the limit in \eqref{eq:bdry} is non-negative,
$$
     \lim_{n\to\infty}  \left|\lim_{R\to\infty}J_{R,\eps_n,T_n}\right|
     =\left|\lim_{n\to\infty}\lim_{R\to\infty}J_{R,\eps_n,T_n}\right|
     =\|f\varphi(1-b)\|_2^2.
$$
Since $\eta_R\uparrow1$ and $(\eps_n,T_n)\uparrow(0,\infty)$, monotone convergence and \eqref{eq:tr} imply
\begin{align*}
     &\int_0^\infty\int_{\mathbb R^N}    a^2\Gamma_{\kappa,t}(u)\,t\,d\omega\,dt
     =\lim_{n\to\infty}\lim_{R\to\infty}I_{R,\eps_n,T_n}\\
     \le\,&2\lim_{n\to\infty}  \left|\lim_{R\to\infty}J_{R,\eps_n,T_n}\right|
       +C\lim_{n\to\infty}\lim_{R\to\infty}K_{R,\eps_n,T_n}\\
      =\,&2\|f\varphi(1-b)\|_2^2   +C\int_0^\infty\int_{\mathbb R^N}
       |u|^2\Psi(v)^2|\nabla_{y,t}v|^2\,t\,d\omega\,dt.
\end{align*}
The constant is independent of $b_0$, so \eqref{eq:loc0} follows.

\medskip
\noindent\textit{Step 4: passage to the indicator datum.}
Let  $b:=\one_{E^c}\in L^2(d\omega)$.  Since $d\omega$ is a Radon measure, choose real-valued
$q_n\in C_c(\mathbb R^N)$ such that
$$
        0\le q_n\le1,   \qquad    \|q_n-b\|_{L^2(d\omega)}\longrightarrow0.
$$
Define
$$
        \widetilde q_n(x):=\frac1{|G|}\sum_{\sigma\in G}q_n(\sigma x).
$$
Then $\widetilde q_n\in C_c(\mathbb R^N)$ is $G$-invariant, $0\le\widetilde q_n\le1$, and
\begin{align*}
     \|\widetilde q_n-b\|_{L^2(d\omega)}  &\le\frac1{|G|}\sum_{\sigma\in G}
       \|q_n\circ\sigma-b\|_{L^2(d\omega)}   =\|q_n-b\|_{L^2(d\omega)}\longrightarrow0.
\end{align*}

Let $\rho_\delta$ be a non-negative radial Euclidean mollifier supported in $B(0,\delta)$, and write
$K_n:=\operatorname{supp}\widetilde q_n$.  For fixed $n$ and $0<\delta<1$,
$$
     \operatorname{supp}(\rho_\delta*\widetilde q_n-\widetilde q_n)   \subset K_n+B(0,1).
$$
Since $d\omega$ is locally finite and $\rho_\delta*\widetilde q_n\to\widetilde q_n$ uniformly,
$$
     \|\rho_\delta*\widetilde q_n-\widetilde q_n\|_{L^2(d\omega)}
     \le\omega(K_n+B(0,1))^{1/2}  \|\rho_\delta*\widetilde q_n-\widetilde q_n\|_\infty
       \longrightarrow0
$$
as $\delta\to 0$.  Choose $0<\delta_n<1/n$ such that
$$
        \|\rho_{\delta_n}*\widetilde q_n-\widetilde q_n\|_{L^2(d\omega)}\le\frac1n,
$$
and let $b_n:=\rho_{\delta_n}*\widetilde q_n$.  By construction,
$b_n\in C_c^\infty(\mathbb R^N)$ and $0\le b_n\le1$.  The radiality of $\rho_{\delta_n}$ and the orthogonality of
the $G$-action ensure that $b_n\circ\sigma=b_n$ for every $\sigma\in G$.
Moreover,
$$
        \|b_n-b\|_{L^2(d\omega)}   \le\frac1n+\|\widetilde q_n-b\|_{L^2(d\omega)}    \longrightarrow0.
$$

Apply \eqref{eq:loc0} to $b_n$, and let
$$
        v_n:=1-P_tb_n,    \qquad    a_n:=\varphi(v_n).
$$
For every $m\ge2$, the $L^2$ contraction of $P_t$ implies
$$
     \int_{1/m}^{m}\int_{\mathbb R^N}|v_n(y,t)-v(y,t)|^2\,d\omega(y)\,dt
     \le(m-m^{-1})\|b_n-b\|_{L^2(d\omega)}^2   \longrightarrow0.
$$
After passing to a diagonal subsequence, $v_n\to v$ a.e.\ on $\mathbb R^{N+1}_+$.
Since $\varphi$ is continuous, $a_n\to a$ a.e.\ on $\mathbb R^{N+1}_+$.
By \eqref{eq:lp2},
\begin{equation*}
        \int_0^\infty\int_{\mathbb R^N}
        \Gamma_{\kappa,t}(u)\,t\,d\omega\,dt<\infty.
\end{equation*}
Since $0\le a_n\le1$, dominated convergence implies
\begin{equation*}
\int_0^\infty\int_{\mathbb R^N}
a_n^2\Gamma_{\kappa,t}(u)\,t\,d\omega\,dt
\longrightarrow
\int_0^\infty\int_{\mathbb R^N}
a^2\Gamma_{\kappa,t}(u)\,t\,d\omega\,dt.
\end{equation*}

For the approximation error, \eqref{eq:lpg} and $v_n-v=-P_t(b_n-b)$ imply
\begin{equation*}
\|\nabla_{y,t}(v_n-v)\|_{L^2(t\,d\omega\,dt)}^2
\le C\|b_n-b\|_{L^2(d\omega)}^2\longrightarrow0.
\end{equation*}
For the limiting datum, \eqref{eq:lpg} gives
\begin{equation*}
\|\nabla_{y,t}v\|_{L^2(t\,d\omega\,dt)}^2
\le C\|b\|_{L^2(d\omega)}^2
=C\omega(E^c)<\infty.
\end{equation*}
Let
\begin{equation*}
        X_n:=\Psi(v_n)\nabla_{y,t}v_n,
        \qquad
        X:=\Psi(v)\nabla_{y,t}v.
\end{equation*}
Then
\begin{align*}
\|X_n-X\|_{L^2(t\,d\omega\,dt)}
&\le\|\Psi\|_\infty
   \|\nabla_{y,t}(v_n-v)\|_{L^2(t\,d\omega\,dt)}
 +\|[\Psi(v_n)-\Psi(v)]\nabla_{y,t}v\|_{L^2(t\,d\omega\,dt)}
\longrightarrow0.
\end{align*}
Indeed, the second term tends to zero by dominated convergence, since $v_n\to v$ almost everywhere and
\begin{equation*}
        |[\Psi(v_n)-\Psi(v)]\nabla_{y,t}v|^2
        \le4\|\Psi\|_\infty^2|\nabla_{y,t}v|^2.
\end{equation*}
The $L^\infty$ contraction of $P_t$ implies
\begin{equation*}
        \|u(X_n-X)\|_{L^2(t\,d\omega\,dt)}
        \le\|f\|_\infty\|X_n-X\|_{L^2(t\,d\omega\,dt)}
        \longrightarrow0.
\end{equation*}
Hence $uX_n\to uX$ in $L^2(t\,d\omega\,dt)$, and
\begin{equation*}
        \|uX_n\|_{L^2(t\,d\omega\,dt)}^2
        \longrightarrow
        \|uX\|_{L^2(t\,d\omega\,dt)}^2.
\end{equation*}

Finally, the Lipschitz continuity of $\varphi$ implies
\begin{equation*}
\|f[\varphi(1-b_n)-\varphi(1-b)]\|_{L^2(d\omega)}
\le\|f\|_\infty\operatorname{Lip}(\varphi)\|b_n-b\|_{L^2(d\omega)}
\longrightarrow0.
\end{equation*}
Since $\varphi(1-b)=\one_E$,
\begin{equation*}
        \|f\varphi(1-b_n)-f\one_E\|_{L^2(d\omega)}
        \longrightarrow0.
\end{equation*}
Applying \eqref{eq:loc0} to $b_n$ and using the preceding convergences,
\begin{align*}
\int_0^\infty\int_{\mathbb R^N}
  a^2\Gamma_{\kappa,t}(u)\,t\,d\omega\,dt
=\,&\lim_{n\to\infty}\int_0^\infty\int_{\mathbb R^N}
  a_n^2\Gamma_{\kappa,t}(u)\,t\,d\omega\,dt\\
\le\,&C\lim_{n\to\infty}\|f\varphi(1-b_n)\|_{L^2(d\omega)}^2
 +C\lim_{n\to\infty}\|uX_n\|_{L^2(t\,d\omega\,dt)}^2\\
=\,&C\int_E|f|^2\,d\omega
 +C\int_0^\infty\int_{\mathbb R^N}
  |u|^2\Psi(v)^2|\nabla_{y,t}v|^2\,t\,d\omega\,dt.
\end{align*}
\end{proof}

\subsection{Estimating the error term}

The support of $\Psi$ confines the error in \eqref{eq:loc} to $\widetilde W$, where $|u|\le\lambda$.

\begin{lemma}\label{lem:err}
Let $E=E_\beta(\lambda)$, $u=P_tf$, $v=P_t\one_E$, and let $\Psi$ be the cut-off fixed above.  Then
\begin{equation*}
\int_0^\infty\int_{\mathbb R^N}
|u(y,t)|^2\Psi(v(y,t))^2|\nabla_{y,t}v(y,t)|^2\,t\,d\omega(y)\,dt
\le C\lambda^2\omega(E^c),
\end{equation*}
where $C$ is structural.
\end{lemma}

\begin{proof}
If $\Psi(v(y,t))\ne0$, then $v(y,t)>c_2$; hence Lemma~\ref{lem:vlo} implies $(y,t)\in\widetilde W$.
Hence some $x\in E$ satisfies $d(x,y)<\beta t$, and
\begin{equation*}
        |u(y,t)|
        =|P_tf(y)|
        \le\mathcal N_P^\beta f(x)
        \le\lambda.
\end{equation*}
By \eqref{eq:bad}, $\|\one_{E^c}\|_{L^2(d\omega)}^2=\omega(E^c)<\infty$.  Moreover,
\begin{equation*}
        v=P_t\one_E=1-P_t\one_{E^c},
        \qquad
        \nabla_{y,t}v=-\nabla_{y,t}P_t\one_{E^c}.
\end{equation*}
Hence \eqref{eq:lpg}, the preceding support bound, and $0\le\Psi\le1$ imply
\begin{align*}
\int_0^\infty\int_{\mathbb R^N}
 |u|^2\Psi(v)^2|\nabla_{y,t}v|^2\,t\,d\omega\,dt
&\le\lambda^2\int_0^\infty\int_{\mathbb R^N}
 \Psi(v)^2|\nabla_{y,t}v|^2\,t\,d\omega\,dt\\
&\le\lambda^2\int_0^\infty\int_{\mathbb R^N}
 |\nabla_{y,t}v|^2\,t\,d\omega\,dt\\
&=\lambda^2\int_0^\infty\int_{\mathbb R^N}
 |\nabla_{y,t}P_t\one_{E^c}|^2\,t\,d\omega\,dt\\
&\le C\lambda^2\|\one_{E^c}\|_{L^2(d\omega)}^2
 =C\lambda^2\omega(E^c).
\end{align*}
\end{proof}

\subsection{Completion of the good-\texorpdfstring{$\lambda$}{lambda} estimate}

Substitute \eqref{eq:loc} into \eqref{eq:glob}, and use Lemma~\ref{lem:err} for the remaining term:
\begin{align}\label{eq:agood}
\int_A\mathcal S_Pf(x)^2\,d\omega(x)
&\le C\int_0^\infty\int_{\mathbb R^N}
  a(y,t)^2\Gamma_{\kappa,t}(u)(y,t)\,t\,d\omega(y)\,dt\\
&\le C\int_E|f(y)|^2\,d\omega(y)
  +C\int_0^\infty\int_{\mathbb R^N}
  |u|^2\Psi(v)^2|\nabla_{y,t}v|^2\,t\,d\omega\,dt\notag\\
&\le C\int_E|f(y)|^2\,d\omega(y)
  +C\lambda^2\omega(E^c).\notag
\end{align}

Since $P_{1/j}f\to f$ in $L^2(d\omega)$, choose a subsequence $P_{1/j_k}f\to f$ almost everywhere.  For this
subsequence, $d(y,y)<\beta/j_k$ and the definition of $\mathcal N_P^\beta$ imply
\begin{equation*}
        |f(y)|
        =\lim_{k\to\infty}|P_{1/j_k}f(y)|
        \le\mathcal N_P^\beta f(y)
        \quad\text{for a.e. }y.
\end{equation*}
Insert \eqref{eq:agood} into \eqref{eq:split}.  The preceding pointwise estimate and
$E=\{\mathcal N_P^\beta f\le\lambda\}$ continue the bound as follows:
\begin{align*}
\omega\{\mathcal S_Pf>\lambda\}
&\le C\omega(E^c)
  +\frac1{\lambda^2}\int_A\mathcal S_Pf(x)^2\,d\omega(x)\\
&\le C\omega(E^c)
  +\frac{C}{\lambda^2}\int_E|f(y)|^2\,d\omega(y)\\
&\le C\omega(E^c)
  +\frac{C}{\lambda^2}\int_E\bigl(\mathcal N_P^\beta f(y)\bigr)^2\,d\omega(y)\\
&=C\omega\{\mathcal N_P^\beta f>\lambda\}
  +\frac{C}{\lambda^2}
   \int_{\{\mathcal N_P^\beta f\le\lambda\}}
   \bigl(\mathcal N_P^\beta f(y)\bigr)^2\,d\omega(y).
\end{align*}
Thus \eqref{eq:main} holds.  For the layer-cake form, let $F$ be a non-negative measurable function and let
$\lambda>0$.  Tonelli's theorem gives
\begin{align}\label{eq:lay}
2\int_0^\lambda s\,\omega\{F>s\}\,ds
&=\int_{\mathbb R^N}\int_0^\lambda
  2s\,\one_{\{F(x)>s\}}\,ds\,d\omega(x)\\
&=\int_{\mathbb R^N}\min\{F(x),\lambda\}^2\,d\omega(x)\notag\\
&=\int_{\{F\le\lambda\}}F(x)^2\,d\omega(x)
  +\lambda^2\omega\{F>\lambda\}.\notag
\end{align}
After dividing \eqref{eq:lay} by $\lambda^2$, we obtain the two-sided comparison
\begin{align*}
    \frac12\left(   \omega\{F>\lambda\}
     +\frac1{\lambda^2}\int_{\{F\le\lambda\}}F(x)^2\,d\omega(x)   \right)
    &\le\omega\{F>\lambda\}   +\frac1{\lambda^2}\int_0^\lambda s\,\omega\{F>s\}\,ds\\
    &=\frac32\omega\{F>\lambda\}   +\frac1{2\lambda^2}\int_{\{F\le\lambda\}}F(x)^2\,d\omega(x)\\
    &\le\frac32\left(  \omega\{F>\lambda\}  +\frac1{\lambda^2}\int_{\{F\le\lambda\}}F(x)^2\,d\omega(x) \right).
\end{align*}
For $F=\mathcal N_P^\beta f$, the middle expression is the right-hand side of \eqref{eq:layer}, apart from its
constant, whereas the first and last expressions are fixed multiples of the right-hand side of \eqref{eq:main}.
Hence \eqref{eq:main} and \eqref{eq:layer} are equivalent after changing $C$.

\section{The chamber-lifted formulation}\label{sec:ch}

We now unfold the full-space theorem on a fixed fundamental chamber.  On chamber representatives, the orbit distance
is Euclidean, while the reflection energy becomes a finite coupling of fiber coordinates across the walls.  Throughout
this section, the estimates are asserted only for $F=Uf$ with $f\in C_c^\infty(\mathbb R^N)$; no independent theorem is
claimed for chamberwise smooth vector data lacking global cross-wall compatibility.

\subsection{The chamber decomposition}

Fix a closed fundamental chamber $\calC\subset\mathbb R^N$, and enumerate
$$
        G=\{\sigma_1,\ldots,\sigma_m\},    \qquad   m=|G|,   \qquad    \sigma_1=\mathrm{Id}.
$$
For $1\le\rho\le m$, let $\Omega_\rho:=\sigma_\rho\calC$, and recall that
$\mathcal W=\bigcup_{\alpha\in R}\alpha^\perp$.
The action of $G$ on the open chambers is simply transitive.  Therefore,
$$
        \mathbb R^N=\bigcup_{\rho=1}^m\Omega_\rho,   \qquad
        \mathbb R^N\setminus\mathcal W    =\bigsqcup_{\rho=1}^m(\Omega_\rho\setminus\mathcal W).
$$
Since $\omega(\mathcal W)=0$, every non-negative measurable function $\Phi$ satisfies
$$
        \int_{\mathbb R^N}\Phi(x)\,d\omega(x)     =\sum_{\rho=1}^m\int_{\Omega_\rho}\Phi(x)\,d\omega(x).
$$

The orbit distance becomes the Euclidean distance on chamber representatives.

\begin{lemma}\label{lem:close}
For every $x,y\in\calC$,
\begin{equation}\label{eq:close}
        d(x,y)=\|x-y\|.
\end{equation}
Consequently, for $1\le\rho,\tau\le m$ and $x,y\in\calC$,
\begin{equation}\label{eq:cdist}
        d(\sigma_\rho x,\sigma_\tau y)=\|x-y\|.
\end{equation}
\end{lemma}

\begin{proof}
By the definition of $d$ and the orthogonality of $G$, we have
$$
        d(x,y)  =\min_{\gamma\in G}\|x-\gamma y\|   =\min_{\gamma\in G}\|\gamma^{-1}x-y\|.
$$
Thus it suffices to prove $\|x-y\|\le\|\sigma x-y\|$ for every $\sigma\in G$.  Fix such a $\sigma$ and choose a
reduced gallery
$$
        \sigma\calC=\calC_0,\calC_1,\ldots,\calC_\ell=\calC.
$$
For $1\le j\le\ell$, choose $\alpha_j\in R$ so that $H_j=\alpha_j^\perp$ is the common wall of
$\calC_{j-1}$ and $\calC_j$ and
$$
        \calC_j=\sigma_{\alpha_j}\calC_{j-1},  \qquad    \langle z,\alpha_j\rangle\ge0    \quad(z\in\calC).
$$
Define
$$
        z_0:=\sigma x,   \qquad    z_j:=\sigma_{\alpha_j}z_{j-1},     \quad 1\le j\le\ell.
$$
Then $z_j\in\calC_j$.  A reduced gallery crosses precisely the walls separating its endpoints;
see \cite[\S 1.6--1.8, pp.~12--16]{Humphreys}.  Hence $H_j$ separates $\calC_{j-1}$ from $\calC$, and the chosen
orientation forces
$$
        \langle z_{j-1},\alpha_j\rangle\le0    \le\langle y,\alpha_j\rangle,    \qquad 1\le j\le\ell.
$$
By simple transitivity of the $G$-action on the chambers, we have
$$
        \bigl(\sigma_{\alpha_\ell}\cdots\sigma_{\alpha_1}\sigma\bigr)\calC=\calC
       \quad  \Longrightarrow   \quad \sigma_{\alpha_\ell}\cdots\sigma_{\alpha_1}\sigma=\mathrm{Id} \quad   \Longrightarrow  \quad  z_\ell=x.
$$
Using \eqref{eq:refl} and the preceding sign condition, for $1\le j\le\ell$ we obtain
\begin{align*}
    \|z_j-y\|^2-\|z_{j-1}-y\|^2  &=\bigg\|z_{j-1}-y  -2\frac{\langle z_{j-1},\alpha_j\rangle}{\|\alpha_j\|^2}\alpha_j\bigg\|^2
     -\|z_{j-1}-y\|^2  =4\frac{\langle z_{j-1},\alpha_j\rangle     \langle y,\alpha_j\rangle}{\|\alpha_j\|^2}   \le0.
\end{align*}
Consequently,
$$
        \|\sigma x-y\|^2     =\|z_0-y\|^2    \ge\|z_1-y\|^2   \ge\cdots
        \ge\|z_\ell-y\|^2    =\|x-y\|^2.
$$
Since $\sigma\in G$ was arbitrary and $\mathrm{Id}\in G$, taking the minimum over $\sigma$ establishes
\eqref{eq:close}.

Finally, as $\gamma$ ranges over $G$, so does $\sigma_\rho^{-1}\gamma\sigma_\tau$.  Hence \eqref{eq:close} implies
\begin{align*}
    d(\sigma_\rho x,\sigma_\tau y)   &=\min_{\gamma\in G}\|\sigma_\rho x-\gamma\sigma_\tau y\|
    =\min_{\gamma\in G}\|x-\sigma_\rho^{-1}\gamma\sigma_\tau y\|    =d(x,y)   =\|x-y\|,
\end{align*}
which is \eqref{eq:cdist}.
\end{proof}

For $x\in\calC$ and $r>0$, define
$$
        B_\calC(x,r):=B(x,r)\cap\calC,   \qquad      V_\calC(x,r):=\omega(B_\calC(x,r)).
$$

The restricted measure has the same ball growth, up to the finite chamber multiplicity.

\begin{lemma}
For every $x\in\calC$ and $r>0$,
\begin{equation}\label{eq:cvol}
        |G|^{-1}V(x,r)\le V_\calC(x,r)\le V(x,r).
\end{equation}
Consequently, $(\calC,\|\cdot\|,d\omega|_\calC)$ is a space of homogeneous type.
\end{lemma}

\begin{proof}
The inclusion $B_\calC(x,r)\subset B(x,r)$ immediately implies
$$
        V_\calC(x,r)\le V(x,r).
$$
For the reverse estimate, fix $\sigma\in G$.  By the $G$-invariance of $d\omega$ and the orthogonality of $\sigma$,
we have
\begin{align*}
    \omega\bigl(B(x,r)\cap\sigma\calC\bigr)   &=\int_\calC\one_{B(x,r)}(\sigma y)\,d\omega(y)
    =\int_\calC\one_{B(\sigma^{-1}x,r)}(y)\,d\omega(y)
    =\omega\bigl(B(\sigma^{-1}x,r)\cap\calC\bigr).
\end{align*}
If $y\in B(\sigma^{-1}x,r)\cap\calC$, then \eqref{eq:cdist} implies
$$
        \|x-y\|    =d(\sigma^{-1}x,y)    \le\|\sigma^{-1}x-y\|   <r.
$$
Therefore,
$$
        B(\sigma^{-1}x,r)\cap\calC      \subset B(x,r)\cap\calC,
$$
and summing over the chambers leads to
\begin{align*}
    V(x,r)  &=\sum_{\sigma\in G}\omega\bigl(B(x,r)\cap\sigma\calC\bigr)
    =\sum_{\sigma\in G}\omega\bigl(B(\sigma^{-1}x,r)\cap\calC\bigr)
    \le\sum_{\sigma\in G}V_\calC(x,r)    =|G|V_\calC(x,r).
\end{align*}
Let $C_D$ be the doubling constant of $d\omega$ on $\mathbb R^N$.  Applying \eqref{eq:cvol} at $r$ and $2r$, we obtain
$$
        V_\calC(x,2r)   \le V(x,2r)   \le C_DV(x,r)    \le C_D|G|V_\calC(x,r).
$$
Since $B_\calC(x,r)$ is the ball of radius $r$ for the restricted Euclidean metric and
$0<V_\calC(x,r)<\infty$, the homogeneous-type conclusion follows.
\end{proof}

\subsection{Chamber lifting and preservation of the \texorpdfstring{$L^p$}{Lp}-norm}

For a measurable scalar-valued function $f$ on $\mathbb R^N$, define
$$
        Uf(x):=\bigl(f(\sigma_1x),f(\sigma_2x),\ldots,f(\sigma_mx)\bigr),     \qquad x\in\calC.
$$
Thus $Uf$ is $\mathbb C^m$-valued.  The lift retains the reflected values as distinct fiber coordinates and is not a
quotient map.

Conversely, let $F=(F_1,\ldots,F_m)$ be measurable on $\calC$.  Since the chambers are disjoint outside $\mathcal W$,
the formula
$$
        (U^{-1}F)(\sigma_\rho x):=F_\rho(x),   \qquad   x\in\calC\setminus\mathcal W,
        \quad 1\le\rho\le m,
$$
defines $U^{-1}F$ uniquely almost everywhere on $\mathbb R^N$; its values on $\mathcal W$ may be chosen arbitrarily.

For $0<p<\infty$, equip $\mathbb C^m$ with the unnormalized norm or quasi-norm
$$
        \|(a_1,\ldots,a_m)\|_{\ell_m^p}^p     :=\sum_{\rho=1}^m|a_\rho|^p,
$$
and, for $p=\infty$, define
$$
        \|(a_1,\ldots,a_m)\|_{\ell_m^\infty}     :=\max_{1\le\rho\le m}|a_\rho|.
$$
For $0<p<\infty$, the $G$-invariance of $d\omega$ and the chamber decomposition imply
\begin{align*}
    \|Uf\|_{L^p(\calC;\ell_m^p)}^p   &=\int_\calC\sum_{\rho=1}^m|f(\sigma_\rho x)|^p\,d\omega(x)
    =\sum_{\rho=1}^m\int_{\sigma_\rho\calC}|f(z)|^p\,d\omega(z)\\
    &=\int_{\mathbb R^N}|f(z)|^p\,d\omega(z)   =\|f\|_{L^p(\mathbb R^N,d\omega)}^p.
\end{align*}
At $p=\infty$, the same argument shows that
\begin{align*}
    \|Uf\|_{L^\infty(\calC;\ell_m^\infty)}   &=\operatorname*{ess\,sup}_{x\in\calC}
      \max_{1\le\rho\le m}|f(\sigma_\rho x)|   =\max_{1\le\rho\le m}
      \operatorname*{ess\,sup}_{z\in\sigma_\rho\calC}|f(z)|\\
    &=\operatorname*{ess\,sup}_{z\in\mathbb R^N}|f(z)|    =\|f\|_{L^\infty(\mathbb R^N,d\omega)}.
\end{align*}
The definitions also show that
$$
        U^{-1}Uf=f\quad\text{a.e.\ on }\mathbb R^N,    \qquad      U(U^{-1}F)=F\quad\text{a.e.\ on }\calC.
$$
Consequently,
$$
        U:L^p(\mathbb R^N,d\omega)    \longrightarrow L^p(\calC,d\omega;\ell_m^p)
$$
is an isometric isomorphism for every $0<p\le\infty$, with the quasi-Banach interpretation when $0<p<1$.
In particular, $U$ is unitary when $p=2$.

Under this identification, the $G$-invariant subspace corresponds exactly to the diagonal subspace of functions
$F=(F_1,\ldots,F_m)$ satisfying
\begin{equation}\label{eq:diag}
        F_1(x)=F_2(x)=\cdots=F_m(x)   \quad\text{for a.e. }x\in\calC.
\end{equation}
One implication follows from the definition of $U$.  Conversely, suppose that $Uf$ satisfies \eqref{eq:diag}.  For
$\sigma\in G$ and almost every $z=\sigma_\rho x\in\Omega_\rho\setminus\mathcal W$, let $\tau$ be the unique index such
that $\sigma\sigma_\rho=\sigma_\tau$.  Then
\begin{align*}
    f(\sigma z)   &=f(\sigma\sigma_\rho x)   =f(\sigma_\tau x)
     =(Uf)_\tau(x)  =(Uf)_\rho(x)  =f(\sigma_\rho x)  =f(z).
\end{align*}
Thus $f$ is $G$-invariant almost everywhere.  The lifted functions used below are not required to belong to the
diagonal subspace.

\subsection{Lifted Poisson extension and chamber energy}

On $L^2(\calC,d\omega;\ell_m^2)$, define the lifted Poisson semigroup by
$$
        \mathbb P_t:=UP_tU^{-1},    \qquad t>0.
$$
Thus, if $F=Uf$, the chamber decomposition and the $G$-invariance of $d\omega$ show, for $x\in\calC$ and
$1\le\rho\le m$,
\begin{align*}
    (\mathbb P_t F)_\rho(x)  &=P_t f(\sigma_\rho x)
    =\sum_{\tau=1}^m\int_{\sigma_\tau\calC}   p_t(\sigma_\rho x,z)f(z)\,d\omega(z)
    =\sum_{\tau=1}^m\int_\calC   p_t(\sigma_\rho x,\sigma_\tau y)F_\tau(y)\,d\omega(y).
\end{align*}
Hence the matrix kernel of $\mathbb P_t$ is
$\bigl(p_t(\sigma_\rho x,\sigma_\tau y)\bigr)_{1\le\rho,\tau\le m}$ on $\calC$.

By \eqref{eq:cdist}, the bound \eqref{eq:psize} applies entrywise with
$d(\sigma_\rho x,\sigma_\tau y)=\|x-y\|$.

Let $F=Uf$, let $u=P_t f$, and write
$$
        \mathbb U(x,t):=\mathbb P_t F(x)     =\bigl(\mathbb U_1(x,t),\ldots,\mathbb U_m(x,t)\bigr).
$$
For each $\alpha\in R_+$, left multiplication by $\sigma_\alpha$ defines a permutation $r_\alpha$ of
$\{1,\ldots,m\}$ by
$$
        \sigma_{r_\alpha(\rho)}=\sigma_\alpha\sigma_\rho,    \qquad 1\le\rho\le m.
$$
If $x\in\calC\setminus\mathcal W$, then $x$ lies on no reflecting hyperplane.  Since
$\sigma_\rho^{-1}\alpha\in R$, it follows that
$$
        \langle\alpha,\sigma_\rho x\rangle   =\langle\sigma_\rho^{-1}\alpha,x\rangle\ne0.
$$
Consequently, for $x\in\calC\setminus\mathcal W$ and $t>0$, define
\begin{align}\label{eq:ceng}
    \mathfrak E_\calC(\mathbb U)(x,t)  &:=\sum_{\rho=1}^m
       \left(|\partial_t\mathbb U_\rho(x,t)|^2+|\nabla_x\mathbb U_\rho(x,t)|^2\right)
     +\sum_{\rho=1}^m\sum_{\alpha\in R_+}\kappa(\alpha)
      \frac{|\mathbb U_\rho(x,t)-\mathbb U_{r_\alpha(\rho)}(x,t)|^2}
       {\langle\alpha,\sigma_\rho x\rangle^2}.
\end{align}
For lifted smooth data, the quotient in the second sum equals
$$
    \frac{\mathbb U_\rho(x,t)-\mathbb U_{r_\alpha(\rho)}(x,t)}   {\langle\alpha,\sigma_\rho x\rangle}
    =\frac{u(\sigma_\rho x,t)-u(\sigma_\alpha\sigma_\rho x,t)}       {\langle\alpha,\sigma_\rho x\rangle},
$$
and has the removable extension described in Section~\ref{sec:pre}.  Since $\omega(\mathcal W)=0$, its values on
$\mathcal W$ do not affect the integral identities below.

For $x\in\calC\setminus\mathcal W$, $t>0$, $1\le\rho\le m$, and $\alpha\in R_+$, the chain rule and the
orthogonality of $\sigma_\rho$ imply
$$
        \partial_t\mathbb U_\rho(x,t)=\partial_tu(\sigma_\rho x,t),   \qquad
        \nabla_x\mathbb U_\rho(x,t)=\sigma_\rho^{-1}\nabla_xu(\sigma_\rho x,t).
$$
By the definition of $r_\alpha$, we also have
$$
        \mathbb U_{r_\alpha(\rho)}(x,t)=u(\sigma_\alpha\sigma_\rho x,t).
$$
These identities yield
\begin{align}\label{eq:elift}
    \sum_{\rho=1}^m\Gamma_{\kappa,t}(u)(\sigma_\rho x,t)  &=\sum_{\rho=1}^m
      \left( \left|\partial_tu(\sigma_\rho x,t)\right|^2+ \left|\nabla_xu(\sigma_\rho x,t) \right|^2\right)
    +\sum_{\rho=1}^m\sum_{\alpha\in R_+}\kappa(\alpha) \frac{|u(\sigma_\rho x,t)-u(\sigma_\alpha\sigma_\rho x,t)|^2}
       {\langle\alpha,\sigma_\rho x\rangle^2}\notag\\ 
       &=\mathfrak E_\calC(\mathbb P_tUf)(x,t). 
\end{align}
Since $\omega(\mathcal W)=0$, \eqref{eq:elift} holds for almost every $x\in\calC$.  In particular, the fiber-difference terms in
\eqref{eq:ceng} are exactly the reflection part of the lifted energy and vanish on the diagonal subspace
\eqref{eq:diag}.

For $\beta>0$ and a lifted datum $F=Uf$, define the chamber non-tangential maximal function by
$$
        \mathbb N_P^\beta F(x)  :=\sup_{\substack{t>0,\ y\in\calC\\ \|x-y\|<\beta t}}
          \max_{1\le\rho\le m} \left|(\mathbb P_t F)_\rho(y)\right|,    \qquad x\in\calC.
$$
The corresponding chamber square function is
$$
        \mathbb S_P F(x)    :=\left(   \int_0^\infty\int_{B_\calC(x,t)} \mathfrak E_\calC(\mathbb P_t F)(y,t)
        \frac{t\,d\omega(y)\,dt}{V_\calC(x,t)}    \right)^{1/2},  \qquad x\in\calC.
$$
Both functions are allowed to take the value $+\infty$.

\subsection{The lifted distribution estimate}

Write $\omega_\calC:=\omega|_\calC$.  The exact maximal-function identity and the square-function comparison below
transfer the full-space estimate to the chamber.

\begin{theorem}\label{thm:ch}
Let $\beta>1$ be the aperture fixed in Theorem~\ref{thm:main}.  There exists a structural constant $C>0$ such that,
whenever $f\in C_c^\infty(\mathbb R^N)$, $F=Uf$, and $\lambda>0$,
\begin{align}\label{eq:cgood}
    \omega_\calC\{x\in\calC:\mathbb S_P F(x)>\lambda\}  &\le C\omega_\calC\{x\in\calC:\mathbb N_P^\beta F(x)>\lambda\}
     +\frac{C}{\lambda^2}  \int_{\{x\in\calC:\mathbb N_P^\beta F(x)\le\lambda\}}
    \bigl(\mathbb N_P^\beta F(x)\bigr)^2\,d\omega(x).
\end{align}
Conversely, if \eqref{eq:cgood} holds uniformly for all such lifts, then \eqref{eq:main} follows.  Hence the two
distribution estimates are equivalent for this class of globally smooth lifts, up to a structural change in $C$.
\end{theorem}

\begin{proof}
Let $m:=|G|$ and $c_G:=m^{-1/2}$.  Fix $1\le\rho\le m$.  First, the chamber decomposition and
\eqref{eq:cdist} show, for $x\in\calC$, that
\begin{align*}
    \mathcal N_P^\beta f(\sigma_\rho x)  &=\sup_{\substack{t>0,\ z\in\mathbb R^N\\d(\sigma_\rho x,z)<\beta t}}
      |P_t f(z)|  =\sup_{\substack{t>0,\ y\in\calC\\\|x-y\|<\beta t}}   \max_{1\le\tau\le m}|(\mathbb P_t F)_\tau(y)|
     =\mathbb N_P^\beta F(x).
\end{align*}

For the square function, the chamber walls are $\omega$-null.  Decomposing the spatial integral and using
\eqref{eq:cdist}, the $G$-invariance of $d\omega$, the identity $V(\sigma_\rho x,t)=V(x,t)$, and
\eqref{eq:elift}, we obtain
\begin{align*}
    \mathcal S_Pf(\sigma_\rho x)^2  &=\sum_{\tau=1}^m\int_0^\infty\int_\calC
      \one_{\{d(\sigma_\rho x,\sigma_\tau y)<t\}}   \Gamma_{\kappa,t}(P_t f)(\sigma_\tau y,t)    \frac{t\,d\omega(y)\,dt}{V(\sigma_\rho x,t)}\\
    &=\int_0^\infty\int_{B_\calC(x,t)}   \sum_{\tau=1}^m\Gamma_{\kappa,t}(P_t f)(\sigma_\tau y,t)
      \frac{t\,d\omega(y)\,dt}{V(x,t)}\\
    &=\int_0^\infty\int_{B_\calC(x,t)}    \mathfrak E_\calC(\mathbb P_t F)(y,t)  \frac{t\,d\omega(y)\,dt}{V(x,t)}.
\end{align*}
Since \eqref{eq:cvol} implies
$$
        V_\calC(x,t)\le V(x,t)\le mV_\calC(x,t),
$$
the preceding identity implies
\begin{equation}\label{eq:scomp}
        c_G\mathbb S_P F(x)  \le\mathcal S_Pf(\sigma_\rho x)  \le\mathbb S_P F(x).
\end{equation}

For $\sigma\in G$, the identities
$$
        d(\sigma x,y)=d(x,y),   \qquad     V(\sigma x,t)=V(x,t)
$$
show that both full-space functionals are $G$-invariant in the cone vertex:
$$
        \mathcal N_P^\beta f(\sigma x)=\mathcal N_P^\beta f(x),   \qquad
        \mathcal S_Pf(\sigma x)=\mathcal S_Pf(x).
$$
Let $M_F:=\mathbb N_P^\beta F$.  The preceding maximal identity and invariance imply
$$
        M_F(x)=\mathcal N_P^\beta f(\sigma_\rho x)    =\mathcal N_P^\beta f(x),    \qquad x\in\calC.
$$
For every $c>0$, the chamber decomposition now shows
\begin{align*}
    \omega\{\mathcal N_P^\beta f>c\}  =m\,\omega_\calC\{M_F>c\},\qquad  \omega\{\mathcal S_Pf>c\}
    =m\,\omega_\calC\{x\in\calC:\mathcal S_Pf(\sigma_\rho x)>c\},
\end{align*}
and
$$
    \int_{\{\mathcal N_P^\beta f\le c\}}(\mathcal N_P^\beta f)^2\,d\omega  =m\int_{\{x\in\calC:M_F(x)\le c\}}M_F(x)^2\,d\omega(x).
$$

By the first inequality in \eqref{eq:scomp}, we have
$$
        \{x\in\calC:\mathbb S_P F(x)>\lambda\}  \subset
        \{x\in\calC:\mathcal S_Pf(\sigma_\rho x)>c_G\lambda\}.
$$
Applying \eqref{eq:main} at the level $c_G\lambda$ and then decomposing the $G$-invariant terms into chambers, we obtain
\begin{align*}
    \omega_\calC\{\mathbb S_P F>\lambda\}  &\le\omega_\calC\{x:\mathcal S_Pf(\sigma_\rho x)>c_G\lambda\}
    =m^{-1}\omega\{\mathcal S_Pf>c_G\lambda\}\\
    &\le Cm^{-1}\omega\{\mathcal N_P^\beta f>c_G\lambda\}  +\frac{C}{mc_G^2\lambda^2}
      \int_{\{\mathcal N_P^\beta f\le c_G\lambda\}}   (\mathcal N_P^\beta f)^2\,d\omega\\
    &=C\omega_\calC\{M_F>c_G\lambda\}   +\frac{C}{c_G^2\lambda^2}
      \int_{\{x\in\calC:M_F(x)\le c_G\lambda\}}M_F(x)^2\,d\omega(x).
\end{align*}
Since $0<c_G\le1$, the two terms at level $c_G\lambda$ satisfy
$$
    \one_{\{M_F>c_G\lambda\}}  \le\one_{\{M_F>\lambda\}}    +\frac{M_F^2}{c_G^2\lambda^2}\one_{\{M_F\le\lambda\}},
    \qquad   \one_{\{M_F\le c_G\lambda\}}\le\one_{\{M_F\le\lambda\}}.
$$
Inserting their integrals into the preceding estimate, we obtain
$$
    \omega_\calC\{\mathbb S_P F>\lambda\}  \le C\omega_\calC\{M_F>\lambda\}  +\frac{C}{\lambda^2}
      \int_{\{x\in\calC:M_F(x)\le\lambda\}}M_F(x)^2\,d\omega(x).
$$
Since $c_G^{-2}=m$, the constant is structural, and \eqref{eq:cgood} follows.

Conversely, assume that \eqref{eq:cgood} holds uniformly for every lift $F=Uf$ with
$f\in C_c^\infty(\mathbb R^N)$.  The second inequality in \eqref{eq:scomp}, the $G$-invariance above, and the maximal
identity imply
\begin{align*}
    \omega\{\mathcal S_Pf>\lambda\}   &=m\omega_\calC\{x:\mathcal S_Pf(\sigma_\rho x)>\lambda\}
    \le m\omega_\calC\{\mathbb S_P F>\lambda\}\\
    &\le Cm\omega_\calC\{M_F>\lambda\}    +\frac{Cm}{\lambda^2} \int_{\{x\in\calC:M_F(x)\le\lambda\}}M_F(x)^2\,d\omega(x)\\
    &=C\omega\{\mathcal N_P^\beta f>\lambda\}   +\frac{C}{\lambda^2}
       \int_{\{\mathcal N_P^\beta f\le\lambda\}}    (\mathcal N_P^\beta f)^2\,d\omega.
\end{align*}
Thus \eqref{eq:main} follows from \eqref{eq:cgood} on the lifted range.
\end{proof}

\section{From orbit cones to Euclidean cones}\label{sec:euc}

\subsection{The orbit and Euclidean maximal functions}

The orbit cone is the union of finitely many Euclidean cones whose vertices lie in the same $G$-orbit.  Accordingly,
the corresponding maximal functions satisfy an exact envelope identity.  Alongside the orbit maximal function introduced in
Section~\ref{sec:loc}, define
$$
        \mathcal N_{P,\mathrm{euc}}^\beta f(x)     :=\sup_{\substack{t>0,\ y\in\mathbb R^N\\\|x-y\|<\beta t}}|P_t f(y)|.
$$

\begin{proposition}
Let $\beta>0$, and let $f$ be a function whose Poisson extension is defined on
$\mathbb R^N\times(0,\infty)$.  Then, for every $x\in\mathbb R^N$,
\begin{equation}\label{eq:nid}
        \mathcal N_P^\beta f(x)   =\max_{\sigma\in G}\mathcal N_{P,\mathrm{euc}}^\beta f(\sigma^{-1}x).
\end{equation}
Consequently, for every $0<p<\infty$,
\begin{equation}\label{eq:np}
        \|\mathcal N_{P,\mathrm{euc}}^\beta f\|_{L^p(d\omega)} \le\|\mathcal N_P^\beta f\|_{L^p(d\omega)}
        \le |G|^{1/p}\|\mathcal N_{P,\mathrm{euc}}^\beta f\|_{L^p(d\omega)},
\end{equation}
where the usual quasi-norm is understood when $0<p<1$.  In particular,
\begin{equation}\label{eq:n1}
        \|\mathcal N_{P,\mathrm{euc}}^\beta f\|_{L^1(d\omega)} \le\|\mathcal N_P^\beta f\|_{L^1(d\omega)}
        \le |G|\|\mathcal N_{P,\mathrm{euc}}^\beta f\|_{L^1(d\omega)}.
\end{equation}
\end{proposition}

\begin{proof}
For $t>0$ and $x,y\in\mathbb R^N$, the orbit condition is equivalent to
$$
        d(x,y)<\beta t  \quad\Longleftrightarrow\quad   \|x-\sigma y\|<\beta t
        \quad\text{for some }\sigma\in G.
$$
Since $G$ is finite and every $\sigma\in G$ is orthogonal, we obtain
\begin{align*}
    \mathcal N_P^\beta f(x)   &=\sup_{\substack{t>0,\ y\in\mathbb R^N\\d(x,y)<\beta t}}|P_t f(y)|
    =\max_{\sigma\in G}   \sup_{\substack{t>0,\ y\in\mathbb R^N\\\|x-\sigma y\|<\beta t}}|P_t f(y)|\\
    &=\max_{\sigma\in G}   \sup_{\substack{t>0,\ y\in\mathbb R^N\\\|\sigma^{-1}x-y\|<\beta t}}|P_t f(y)|
    =\max_{\sigma\in G}\mathcal N_{P,\mathrm{euc}}^\beta f(\sigma^{-1}x),
\end{align*}
which is \eqref{eq:nid}.

The term $\sigma=\mathrm{Id}$ in \eqref{eq:nid} implies
$$
        \mathcal N_{P,\mathrm{euc}}^\beta f(x)\le\mathcal N_P^\beta f(x).
$$
For $0<p<\infty$, the reverse bound follows from \eqref{eq:nid} and the $G$-invariance of $d\omega$:
\begin{align*}
    \|\mathcal N_P^\beta f\|_{L^p(d\omega)}^p  &=\int_{\mathbb R^N}
      \max_{\sigma\in G}\bigl(\mathcal N_{P,\mathrm{euc}}^\beta f(\sigma^{-1}x)\bigr)^p\,d\omega(x)\\
    &\le\sum_{\sigma\in G}\int_{\mathbb R^N}   \bigl(\mathcal N_{P,\mathrm{euc}}^\beta f(\sigma^{-1}x)\bigr)^p\,d\omega(x)\\
    &=\sum_{\sigma\in G}\int_{\mathbb R^N}   \bigl(\mathcal N_{P,\mathrm{euc}}^\beta f(z)\bigr)^p\,d\omega(z)\\
    &=|G|\|\mathcal N_{P,\mathrm{euc}}^\beta f\|_{L^p(d\omega)}^p.
\end{align*}
Taking the $p$-th root establishes \eqref{eq:np}, whose specialization to $p=1$ is \eqref{eq:n1}.
\end{proof}

\subsection{The Euclidean-cone area function}

For comparison with the Euclidean non-tangential maximal function, define
\begin{equation}\label{eq:seuc}
    S_{P,\mathrm{euc}}f(x)  :=\left( \int_0^\infty\int_{\|x-y\|<t}  \Gamma_{\kappa,t}(P_t f)(y,t)\,
    \frac{t\,d\omega(y)\,dt}{V(x,t)}   \right)^{1/2}.
\end{equation}
Thus only the cone in \eqref{eq:area} is changed; the intrinsic Dunkl energy is unchanged.

Since $\mathrm{Id}\in G$, we have
$$
        \|x-y\|<t   \quad\Longrightarrow\quad     d(x,y)\le\|x-y\|<t.
$$
Thus the Euclidean cone is contained in the orbit cone.  Since $\Gamma_{\kappa,t}\ge0$ and the denominators agree,
whenever both sides are defined in $[0,\infty]$,
\begin{align}\label{eq:sdom}
    S_{P,\mathrm{euc}}f(x)   &=\left(\int_0^\infty\int_{\|x-y\|<t}
      \Gamma_{\kappa,t}(P_t f)(y,t)\frac{t\,d\omega(y)\,dt}{V(x,t)} \right)^{1/2}\notag\\
    &\le\left(  \int_0^\infty\int_{d(x,y)<t}   \Gamma_{\kappa,t}(P_t f)(y,t)\frac{t\,d\omega(y)\,dt}{V(x,t)}  \right)^{1/2}
     =\mathcal S_Pf(x).
\end{align}

\subsection{Integrating the good-\texorpdfstring{$\lambda$}{lambda} inequality}

\begin{proposition}
Let $\beta>1$ be the structural aperture fixed in Theorem~\ref{thm:main}.  Then there exists a structural constant
$C>0$ such that, for every $f\in C_c^\infty(\mathbb R^N)$,
\begin{equation}\label{eq:l1}
        \|\mathcal S_Pf\|_{L^1(d\omega)}   \le C\|\mathcal N_P^\beta f\|_{L^1(d\omega)}.
\end{equation}
Consequently,
$$
        \|S_{P,\mathrm{euc}}f\|_{L^1(d\omega)}    \le C\|\mathcal N_{P,\mathrm{euc}}^\beta f\|_{L^1(d\omega)}.
$$
Both estimates are understood in $[0,\infty]$.
\end{proposition}

\begin{proof}
Let $C_0$ be the structural constant in \eqref{eq:layer}, and let $F:=\mathcal N_P^\beta f$.
If $\|F\|_{L^1(d\omega)}=\infty$, then \eqref{eq:l1} is immediate.  Otherwise, the layer-cake formula,
\eqref{eq:layer}, and Tonelli's theorem give
\begin{align*}
    \|\mathcal S_Pf\|_{L^1(d\omega)}  &=\int_0^\infty\omega\{\mathcal S_Pf>\lambda\}\,d\lambda\\
    &\le C_0\int_0^\infty\omega\{F>\lambda\}\,d\lambda    +C_0\int_0^\infty\lambda^{-2}
      \int_0^\lambda s\,\omega\{F>s\}\,ds\,d\lambda\\
    &=C_0\|F\|_{L^1(d\omega)}  +C_0\int_0^\infty s\,\omega\{F>s\}  \left(\int_s^\infty\lambda^{-2}\,d\lambda\right)ds\\
    &=C_0\|F\|_{L^1(d\omega)}   +C_0\int_0^\infty\omega\{F>s\}\,ds    =2C_0\|F\|_{L^1(d\omega)}.
\end{align*}
Hence \eqref{eq:l1} holds with $C=2C_0$.

Finally, \eqref{eq:sdom}, \eqref{eq:l1}, and \eqref{eq:n1} imply
\begin{align*}
    \|S_{P,\mathrm{euc}}f\|_{L^1(d\omega)}  &\le\|\mathcal S_Pf\|_{L^1(d\omega)}
     \le C\|\mathcal N_P^\beta f\|_{L^1(d\omega)}
    \le C|G|\|\mathcal N_{P,\mathrm{euc}}^\beta f\|_{L^1(d\omega)}.
\end{align*}
Absorbing the structural factor $|G|$ into $C$ proves the Euclidean-cone estimate.
\end{proof}

\section{Hardy-space and \texorpdfstring{$L^p$}{Lp} consequences}\label{sec:h1}

\subsection{The Hardy-space endpoint}

Atomic density and lower semicontinuity extend the smooth endpoint estimate to the Poisson maximal Hardy space.
Following \cite[(2.7)]{ADH}, define
$$
        \mathcal M_Pf(x)    =\sup_{\|x-y\|<t}|P_tf(y)|     =\mathcal N_{P,\mathrm{euc}}^1f(x),
$$
and
$$
        H^1_{\mathrm{max},P}     =\bigl\{f\in L^1(d\omega):\mathcal M_Pf\in L^1(d\omega)\bigr\}.
$$
Equip this space with the norm
$$
        \|f\|_{H^1_{\mathrm{max},P}}    =\|\mathcal M_Pf\|_{L^1(d\omega)}.
$$
For $f\in H^1_{\mathrm{max},P}$, apply \cite[Lemma~10.2]{ADH} to $u(t,y)=P_tf(y)$ with $a=\beta$ and $b=1$.
Then
\begin{align}\label{eq:ap}
    \|\mathcal N_{P,\mathrm{euc}}^\beta f\|_{L^1(d\omega)}
    &\le C(\beta+1)^{\mathbf N}    \|\mathcal N_{P,\mathrm{euc}}^1f\|_{L^1(d\omega)}\notag\\
    &=C(\beta+1)^{\mathbf N}\|f\|_{H^1_{\mathrm{max},P}}.
\end{align}

Recall that a $(1,2)$-atom is supported in a Euclidean ball $B$, has integral zero, and satisfies
$\|a\|_{L^2(d\omega)}\le\omega(B)^{-1/2}$.  Let
$$
        \mathcal D_0     =\left\{\phi\in C_c^\infty(\mathbb R^N):
        \int_{\mathbb R^N}\phi\,d\omega=0\right\}.
$$
If $0\ne\phi\in\mathcal D_0$ and $\operatorname{supp}\phi\subset B$, then
$$
        a_\phi:=\frac{\phi}{\omega(B)^{1/2}\|\phi\|_{L^2(d\omega)}}.
$$
The support satisfies $\operatorname{supp}a_\phi\subset B$.  Since $\phi\in\mathcal D_0$, its cancellation is
$\int_{\mathbb R^N}a_\phi\,d\omega=0$.  Its normalization is
$\|a_\phi\|_{L^2(d\omega)}=\omega(B)^{-1/2}$.
Thus $a_\phi$ is a $(1,2)$-atom and
$\phi=\omega(B)^{1/2}\|\phi\|_{L^2(d\omega)}a_\phi$.  Hence every element of $\mathcal D_0$ is a scalar multiple of a
$(1,2)$-atom.

\begin{lemma}\label{lem:dense}
The space $\mathcal D_0$ is dense in $H^1_{\mathrm{max},P}$.
\end{lemma}

\begin{proof}
By \cite[Theorems~2.1 and~2.2]{ADH}, $H^1_{\mathrm{max},P}$ coincides with the Dunkl heat-maximal Hardy space.
Theorem~1.6 of \cite{DH1}, applied with $q=2$, then identifies this space with
$$
        H^1_{\mathrm{max},P}   =H^1_{(1,2)}(\mathbb R^N,\|\cdot\|,d\omega).
$$
The norms are equivalent, so it suffices to approximate each $(1,2)$-atom by elements of $\mathcal D_0$.

Let $a$ be such an atom associated with a Euclidean ball $B$.
Choose $h_j\in C_c^\infty(2B)$ such that $h_j\to a$ in $L^2(d\omega)$, and fix
$\psi\in C_c^\infty(2B)$ with $\int\psi\,d\omega=1$.  Let
$$
        c_j=\int_{\mathbb R^N}h_j\,d\omega,   \qquad    a_j=h_j-c_j\psi.
$$
Then $a_j\in\mathcal D_0$.  Moreover, the cancellation of $a$ implies
$$
    |c_j|   =\left|\int_{2B}(h_j-a)\,d\omega\right|   \le\omega(2B)^{1/2}\|h_j-a\|_{L^2(d\omega)}.
$$
Substituting this bound, we obtain
\begin{align*}
    \|a_j-a\|_{L^2(d\omega)}  &\le\|h_j-a\|_{L^2(d\omega)}+|c_j|\|\psi\|_{L^2(d\omega)}
    \le\bigl(1+\omega(2B)^{1/2}\|\psi\|_{L^2(d\omega)}\bigr)     \|h_j-a\|_{L^2(d\omega)}   \longrightarrow0.
\end{align*}
Since $a_j-a$ is supported in $2B$ and has integral zero, normalization as a $(1,2)$-atom shows that
$$
        \|a_j-a\|_{H^1_{(1,2)}}  \le\omega(2B)^{1/2}\|a_j-a\|_{L^2(d\omega)}   \longrightarrow0.
$$
For a finite atomic sum $g=\sum_{\ell=1}^L\lambda_\ell a_\ell$, choose the preceding approximants
$a_{\ell,j}\in\mathcal D_0$.  Then
$$
        \bigg\|\sum_{\ell=1}^L\lambda_\ell a_{\ell,j}-g\bigg\|_{H^1_{(1,2)}}    \le\sum_{\ell=1}^L|\lambda_\ell|
        \|a_{\ell,j}-a_\ell\|_{H^1_{(1,2)}}    \longrightarrow0.
$$
Finite atomic sums are dense in $H^1_{(1,2)}$, so the preceding approximation proves the density of $\mathcal D_0$.
\end{proof}

Let $f\in L^1(d\omega)$.  For $m\in\mathbb N_0$, $\nu\in\mathbb N_0^N$, and
$D=\partial_t^m\partial_y^\nu$, \cite[Proposition~5.1(c)]{ADH} and \eqref{eq:psize} imply
\begin{align*}
    |Dp_t(y,z)|  &\le Ct^{-m-|\nu|}p_t(y,z)  \le\frac{C}{t^{m+|\nu|}V(y,z,t+d(y,z))}\,
       \frac{t}{t+d(y,z)}    \le\frac{C}{t^{m+|\nu|}V(y,t)}.
\end{align*}
Fix $K\Subset\mathbb R^{N+1}_+$.  Since $t\ge\delta_K>0$ on $K$, \cite[(3.1), p.~7]{ADH} gives
$$
        V(y,t)\ge c\,t^{\mathbf N}\ge c\,\delta_K^{\mathbf N},     \qquad (y,t)\in K.
$$
Therefore,
$$
        \sup_{(y,t)\in K}\int_{\mathbb R^N}|Dp_t(y,z)f(z)|\,d\omega(z)
        \le C_{D,K}\|f\|_{L^1(d\omega)}<\infty,
$$
and differentiation under the integral shows that $P_tf$ is smooth on $\mathbb R^{N+1}_+$.
We henceforth define $\mathcal S_Pf$ and $S_{P,\mathrm{euc}}f$ by \eqref{eq:area} and \eqref{eq:seuc},
respectively, as measurable functions with values in $[0,\infty]$.

The passage to $H^1_{\mathrm{max},P}$ rests on the following lower semicontinuity.

\begin{lemma}\label{lem:lsc}
Let $\mathfrak S$ denote either $\mathcal S_P$ or $S_{P,\mathrm{euc}}$.  If $f_j\to f$ in $L^1(d\omega)$, then
$$
        \mathfrak Sf(x)\le\liminf_{j\to\infty}\mathfrak Sf_j(x),     \qquad x\in\mathbb R^N.
$$
\end{lemma}

\begin{proof}
For $y\notin\mathcal W$, let
$$
    \mathbf D_{\kappa,t}U(y,t)  =\left(  \partial_tU(y,t),   \nabla_yU(y,t),
    \left(  \sqrt{\kappa(\alpha)}  \frac{U(y,t)-U(\sigma_\alpha y,t)}{\langle\alpha,y\rangle}  \right)_{\alpha\in R_+}  \right).
$$
Then
$$
        \Gamma_{\kappa,t}(U)(y,t)    =|\mathbf D_{\kappa,t}U(y,t)|^2.
$$
For
$D\in\{\mathrm{Id},\partial_t,\partial_{y_1},\ldots,\partial_{y_N}\}$,
the preceding kernel bound implies
\begin{align*}
    |DP_t(f_j-f)(y)|  &=\left|\int_{\mathbb R^N}Dp_t(y,z)(f_j-f)(z)\,d\omega(z)\right|
    \le\sup_{z\in\mathbb R^N}|Dp_t(y,z)|\,\|f_j-f\|_{L^1(d\omega)}    \longrightarrow0.
\end{align*}
Applying the same estimate with $D=\mathrm{Id}$ at $\sigma_\alpha y$ also shows, for $y\notin\mathcal W$, that
$$
        \frac{P_tf_j(y)-P_tf_j(\sigma_\alpha y)}{\langle\alpha,y\rangle}   \longrightarrow
        \frac{P_tf(y)-P_tf(\sigma_\alpha y)}{\langle\alpha,y\rangle}.
$$
Since $\omega(\mathcal W)=0$, it follows that
$$
        \mathbf D_{\kappa,t}P_tf_j(y)\longrightarrow\mathbf D_{\kappa,t}P_tf(y)
$$
for $(d\omega\otimes dt)$-almost every $(y,t)$.

For fixed $x$, let $\Omega_x(t)=\{y:d(x,y)<t\}$ for $\mathfrak S=\mathcal S_P$ and
$\Omega_x(t)=\{y:\|x-y\|<t\}$ for $\mathfrak S=S_{P,\mathrm{euc}}$.
Fatou's lemma now implies
\begin{align*}
    \mathfrak Sf(x)^2   &=\int_0^\infty\int_{\Omega_x(t)}   \left|\mathbf D_{\kappa,t}P_tf(y) \right|^2
      \frac{t\,d\omega(y)\,dt}{V(x,t)}\\
    &\le\liminf_{j\to\infty}   \int_0^\infty\int_{\Omega_x(t)}   \left|\mathbf D_{\kappa,t}P_tf_j(y)\right|^2
      \frac{t\,d\omega(y)\,dt}{V(x,t)}   =\liminf_{j\to\infty}\mathfrak Sf_j(x)^2.
\end{align*}
Taking square roots establishes the asserted lower semicontinuity.
\end{proof}

\begin{theorem}\label{thm:h1}
Let $\beta>1$ be the aperture fixed in Theorem~\ref{thm:main}.  Then there exists $C>0$ such that, for every
$f\in H^1_{\mathrm{max},P}$, the intrinsic area functionals $\mathcal S_Pf$ and $S_{P,\mathrm{euc}}f$ belong to
$L^1(d\omega)$ and satisfy
$$
        \|\mathcal S_Pf\|_{L^1(d\omega)}    +\|S_{P,\mathrm{euc}}f\|_{L^1(d\omega)}
        \le C\|f\|_{H^1_{\mathrm{max},P}}.
$$
\end{theorem}

\begin{proof}
For $h\in\mathcal D_0$, \eqref{eq:sdom}, \eqref{eq:l1}, \eqref{eq:n1}, and \eqref{eq:ap} imply
\begin{align*}
    \|\mathcal S_Ph\|_{L^1(d\omega)}  +\|S_{P,\mathrm{euc}}h\|_{L^1(d\omega)}
    &\le2\|\mathcal S_Ph\|_{L^1(d\omega)}   \le C\|\mathcal N_P^\beta h\|_{L^1(d\omega)}\\
    &\le C|G|\|\mathcal N_{P,\mathrm{euc}}^\beta h\|_{L^1(d\omega)}\\
    &\le C|G|(\beta+1)^{\mathbf N}\|h\|_{H^1_{\mathrm{max},P}}   \le C\|h\|_{H^1_{\mathrm{max},P}}.
\end{align*}
Here the last constant is structural because $\beta$ is fixed.

Let $f\in H^1_{\mathrm{max},P}$.  By Lemma~\ref{lem:dense}, choose $f_n\in\mathcal D_0$ such that
$$
        \|f_n-f\|_{H^1_{\mathrm{max},P}}\longrightarrow0.
$$
By the non-tangential boundary convergence in \cite[Corollary~5.4]{ADH}, we have
$$
        |h(x)|\le\mathcal M_Ph(x)
        \qquad\text{for almost every $x$ and every $h\in H^1_{\mathrm{max},P}$}.
$$
Applying the boundary estimate to $h=f_n-f$,
$$
        \|f_n-f\|_{L^1(d\omega)}  \le\|\mathcal M_P(f_n-f)\|_{L^1(d\omega)}
        =\|f_n-f\|_{H^1_{\mathrm{max},P}}    \longrightarrow0.
$$
Lemma~\ref{lem:lsc} implies, pointwise,
\begin{align*}
    \mathcal S_Pf(x)+S_{P,\mathrm{euc}}f(x)  &\le\liminf_{n\to\infty}\mathcal S_Pf_n(x)
        +\liminf_{n\to\infty}S_{P,\mathrm{euc}}f_n(x)\\
    &\le\liminf_{n\to\infty}  \bigl(\mathcal S_Pf_n(x)+S_{P,\mathrm{euc}}f_n(x)\bigr).
\end{align*}
Fatou's lemma and the smooth estimate now give
\begin{align*}
    \|\mathcal S_Pf\|_{L^1(d\omega)}   +\|S_{P,\mathrm{euc}}f\|_{L^1(d\omega)}
    \le\,&\int_{\mathbb R^N}\liminf_{n\to\infty}
     \bigl(\mathcal S_Pf_n(x)+S_{P,\mathrm{euc}}f_n(x)\bigr)\,d\omega(x)\\
    \le\,&\liminf_{n\to\infty}  \bigl(\|\mathcal S_Pf_n\|_{L^1(d\omega)}
     +\|S_{P,\mathrm{euc}}f_n\|_{L^1(d\omega)}\bigr)\\
     \le\,& C\lim_{n\to\infty}\|f_n\|_{H^1_{\mathrm{max},P}}   =C\|f\|_{H^1_{\mathrm{max},P}}.
\end{align*}
\end{proof}

Theorem~\ref{thm:h1} supplies the maximal-to-area bound.  The converse norm estimate is recalled next.

\begin{corollary}\label{cor:h1}
For $f\in L^1(d\omega)$, the following conditions are equivalent:
\begin{enumerate}[(i)]
\item $f\in H^1_{\mathrm{max},P}$;
\item $S_{P,\mathrm{euc}}f\in L^1(d\omega)$;
\item $\mathcal S_Pf\in L^1(d\omega)$.
\end{enumerate}
Whenever these conditions hold,
$$
        \|f\|_{H^1_{\mathrm{max},P}}   \simeq\|S_{P,\mathrm{euc}}f\|_{L^1(d\omega)}
        \simeq\|\mathcal S_Pf\|_{L^1(d\omega)}.
$$
\end{corollary}

\begin{proof}
For $f\in L^1(d\omega)\cap L^2(d\omega)$, the spectral calculus gives
$$
        Q_tf:=t\sqrt{-\Delta_\kappa}\,P_tf=-t\partial_tP_tf.
$$
For each $t>0$, \cite[the paragraph preceding (2.9)]{ADH} extends $Q_t$ boundedly to $L^1(d\omega)$.
Independently, \cite[Proposition~5.1(c)]{ADH}, symmetry, and $P_t1=1$ imply
$$
    \sup_{z\in\mathbb R^N}\int_{\mathbb R^N}|t\partial_tp_t(y,z)|\,d\omega(y)
    \le C\sup_{z\in\mathbb R^N}\int_{\mathbb R^N}p_t(y,z)\,d\omega(y)  =C.
$$
Thus $-t\partial_tP_t$ is bounded on $L^1(d\omega)$.  The two bounded operators agree on the dense subspace
$L^1(d\omega)\cap L^2(d\omega)$ and therefore on all of $L^1(d\omega)$.  We henceforth use the jointly measurable
representative
$$
        Q_tf(y)=-t\partial_tP_tf(y),   \qquad (y,t)\in\mathbb R^N\times(0,\infty).
$$
Following \cite[(2.9)]{ADH}, define
$$
        S_Qf(x)=\left(  \int_0^\infty\int_{\|x-y\|<t}   |Q_tf(y)|^2\frac{d\omega(y)\,dt}{tV(x,t)}  \right)^{1/2}.
$$
Since the vertical component is part of $\Gamma_{\kappa,t}$, we have
\begin{align}\label{eq:sq}
    S_Qf(x)  &=\left( \int_0^\infty\int_{\|x-y\|<t}   |\partial_tP_tf(y)|^2   \frac{t\,d\omega(y)\,dt}{V(x,t)} \right)^{1/2}\notag\\
    &\le\left(  \int_0^\infty\int_{\|x-y\|<t} \Gamma_{\kappa,t}(P_tf)(y,t) \frac{t\,d\omega(y)\,dt}{V(x,t)}  \right)^{1/2}\notag\\
    &=S_{P,\mathrm{euc}}f(x)   \le\mathcal S_Pf(x).
\end{align}

By \cite[Theorem~2.1 and (2.7)]{ADH}, $H^1_{\mathrm{max},P}$ is the semigroup Hardy space used there;
\cite[Theorem~2.3]{ADH} characterizes it by $S_Q$.  Hence
$$
        f\in H^1_{\mathrm{max},P}    \quad\Longleftrightarrow\quad   S_Qf\in L^1(d\omega).
$$
Moreover,
$$
        \|f\|_{H^1_{\mathrm{max},P}}     \simeq\|S_Qf\|_{L^1(d\omega)}.
$$
Together with \eqref{eq:sq}, Theorem~\ref{thm:h1} now establishes
$$
        \mathrm{(i)}   \Longrightarrow\mathrm{(iii)}
        \Longrightarrow\mathrm{(ii)}    \Longrightarrow S_Qf\in L^1(d\omega)       \Longrightarrow\mathrm{(i)}.
$$
Under these equivalent conditions, we have
$$
        \|f\|_{H^1_{\mathrm{max},P}}     \lesssim\|S_Qf\|_{L^1(d\omega)}
        \le\|S_{P,\mathrm{euc}}f\|_{L^1(d\omega)}   \le\|\mathcal S_Pf\|_{L^1(d\omega)}
        \lesssim\|f\|_{H^1_{\mathrm{max},P}}.
$$
\end{proof}

\subsection{The \texorpdfstring{$L^p$}{Lp} estimate and scope}

Integrating the distribution estimate yields the full range $0<p<2$.

\begin{corollary}
Let $\beta>1$ be the aperture fixed in Theorem~\ref{thm:main}, and let $0<p<2$.
Then there exists $C_p>0$ such that, for every $f\in C_c^\infty(\mathbb R^N)$,
\begin{equation}\label{eq:lp}
        \|\mathcal S_Pf\|_{L^p(d\omega)}     \le C_p\|\mathcal N_P^\beta f\|_{L^p(d\omega)}.
\end{equation}
For $F=Uf$, the equivalent chamber estimate is
$$
        \|\mathbb S_P(Uf)\|_{L^p(\calC,d\omega)}     \le C_p\|\mathbb N_P^\beta(Uf)\|_{L^p(\calC,d\omega)}.
$$
\end{corollary}

\begin{proof}
Let $F=\mathcal N_P^\beta f$.  We may assume that $F\in L^p(d\omega)$.
Since all integrands below are non-negative, the layer-cake formula, \eqref{eq:layer}, and Tonelli's theorem give
\begin{align*}
    \|\mathcal S_Pf\|_{L^p(d\omega)}^p   &=p\int_0^\infty\lambda^{p-1}   \omega\{\mathcal S_Pf>\lambda\}\,d\lambda\\
    &\le Cp\int_0^\infty\lambda^{p-1}\omega\{F>\lambda\}\,d\lambda   +Cp\int_0^\infty\lambda^{p-3}
         \int_0^\lambda s\,\omega\{F>s\}\,ds\,d\lambda\\
    &=C\|F\|_{L^p(d\omega)}^p    +Cp\int_0^\infty s\,\omega\{F>s\}    \int_s^\infty\lambda^{p-3}\,d\lambda\,ds\\
    &=C\|F\|_{L^p(d\omega)}^p  +\frac{Cp}{2-p}\int_0^\infty s^{p-1}\omega\{F>s\}\,ds\\
    &=C\left(1+\frac1{2-p}\right)\|F\|_{L^p(d\omega)}^p.
\end{align*}
Taking the $p$-th root establishes \eqref{eq:lp}.

Finally, the maximal-function identity in the proof of Theorem~\ref{thm:ch} and the $G$-invariance of the
full-space maximal function imply
$$
    \|\mathcal N_P^\beta f\|_{L^p(d\omega)}   =|G|^{1/p}\|\mathbb N_P^\beta(Uf)\|_{L^p(\calC,d\omega)}.
$$
Separately, \eqref{eq:scomp} and the $G$-invariance of the full-space square function imply
$$
    |G|^{1/p-1/2}\|\mathbb S_P(Uf)\|_{L^p(\calC,d\omega)}
    \le\|\mathcal S_Pf\|_{L^p(d\omega)}
     \le|G|^{1/p}\|\mathbb S_P(Uf)\|_{L^p(\calC,d\omega)}.
$$
Therefore \eqref{eq:lp} implies
\begin{align*}
    \|\mathbb S_P(Uf)\|_{L^p(\calC,d\omega)}  &\le|G|^{1/2-1/p}\|\mathcal S_Pf\|_{L^p(d\omega)}\\
    &\le C_p|G|^{1/2-1/p}\|\mathcal N_P^\beta f\|_{L^p(d\omega)}\\
    &=C_p|G|^{1/2}\|\mathbb N_P^\beta(Uf)\|_{L^p(\calC,d\omega)}.
\end{align*}
Conversely, the chamber estimate and the same comparisons show that
\begin{align*}
    \|\mathcal S_Pf\|_{L^p(d\omega)}  &\le|G|^{1/p}\|\mathbb S_P(Uf)\|_{L^p(\calC,d\omega)}\\
    &\le C_p|G|^{1/p}\|\mathbb N_P^\beta(Uf)\|_{L^p(\calC,d\omega)}\\
    &=C_p\|\mathcal N_P^\beta f\|_{L^p(d\omega)}.
\end{align*}
These two norm comparisons establish the equivalence for globally smooth lifts.
\end{proof}

The intrinsic energy contains the vertical derivative, the horizontal gradient, and the reflection-difference term.
The good-$\lambda$ estimate controls all three components.  The converse in
Corollary~\ref{cor:h1} instead follows from the known $Q_t$-square-function characterization and
\eqref{eq:sq}.  Accordingly, the converse is only a norm implication; no reverse good-$\lambda$ inequality is
asserted.

\raggedbottom
%
%

\bigskip
\noindent\textbf{Acknowledgements:}
Yanchang Han is supported by the National Natural Science Foundation of China (Grant No.
12471097) and the Guangdong Province Natural Science Foundation (Grant No.
2024A1515013107).
Ji Li is supported by the Australian Research Council (DP 220100285 and DP 260100485).
Liangchuan Wu is supported by the National Natural Science Foundation of China (Grant No. 12201002).

\end{document}